\documentclass[final]{siamltex}
\usepackage{latexsym,amsbsy,amssymb,amsmath}
\usepackage{color}
\usepackage{subcaption}
\usepackage{graphicx}
\usepackage{psfrag}
\usepackage{float}
\usepackage{t1enc}
\usepackage{amsfonts}
\usepackage{stmaryrd}
\usepackage{multirow}
\usepackage[table]{xcolor}

\usepackage{epstopdf}
\usepackage{graphics}
\usepackage{epsf}
\usepackage{xfrac}
\usepackage{amscd}
\usepackage{wrapfig}
\usepackage[mathscr]{euscript}
\usepackage[english]{babel}
\usepackage{epsfig}
\usepackage{dsfont}
\everymath{\displaystyle}
\usepackage{setspace}

\newtheorem{remark}[theorem]{Remark}

\newcommand{\bI}{\pmb{I}}

\newcommand{\bv}{\pmb{v}}
\newcommand{\be}{\pmb{e}}

\newcommand{\bw}{\pmb{w}}

\newcommand{\bA}{\pmb{A}}

\newcommand{\bF}{\pmb{F}}

\newcommand{\bM}{\pmb{M}}

\newcommand{\bu}{\pmb{u}}
\newcommand{\uubu}{\underline{\underline{\pmb{u}}}}

\newcommand{\mR}{\mathbb{R}}

\definecolor{OliveGreen}{rgb}{0,0.6,0}

\title{Overlapping Localized Exponential Time Differencing Methods for Diffusion Problems\footnotemark[1]}

\author{Thi-Thao-Phuong Hoang\footnotemark[2]
\and Lili Ju\footnotemark[2]\ 
\and Zhu Wang\footnotemark[2]}

\begin{document}

\maketitle

\renewcommand{\thefootnote}{\fnsymbol{footnote}}
\footnotetext[1]{This work is partially supported by US Department of Energy under grant number DE-SC0016540 and US National Science Foundation under grant number  DMS-1521965.}
\footnotetext[2]{Department of Mathematics and Interdisciplinary Mathematics Institute, University of South Carolina, Columbia, SC 29208, USA. 
Email: \texttt{hoang5@mailbox.sc.edu}, \texttt{ju@math.sc.edu}, \texttt{wangzhu@math.sc.edu}.}
\renewcommand{\thefootnote}{\arabic{footnote}}

\maketitle

\begin{abstract}
The paper is concerned with overlapping domain decomposition and exponential time differencing for the diffusion equation discretized in space by cell-centered finite differences. Two localized  exponential time differencing methods are proposed to solve the fully discrete problem: the first method is based on Schwarz iteration applied at each time step and involves solving stationary problems in the subdomains at each iteration, while the second method is based on the Schwarz waveform relaxation algorithm in which time-dependent subdomain problems are solved at each iteration. The convergence of the associated iterative solutions to the corresponding 
fully discrete multidomain solution and to the exact semi-discrete solution is rigorously proved.  Numerical experiments are carried out to confirm theoretical results and to compare the performance of the two methods. 
\end{abstract} 


\pagestyle{myheadings}

\section{Introduction}
Exponential time differencing (ETD) methods are numerical methods for the time integration of systems of evolutionary partial differential equations based on exponential integrators and the variation-of-constants formula. The methods have been studied by many researchers for various classes of problems, for instance, see \cite{hochbruck98, Cox2002, Du2005, Krogstad2005, Nie2006, Tokman2006, Nie2008, hockbruck2009} and the references therein. A sound review in this direction and additional references are given in \cite{hochbruckActa}. Except for preservation of the system's exponential behavior in the discrete sense, one of the most important properties of these methods is that large time steps can be used for stiff problems without affecting stability of the solution, while explicit methods often require tiny time step sizes, which is often very expensive in terms of computational cost. 

Due to the development of supercomputers and parallel computing technologies, numerical methods based on domain decomposition (DD) have attracted great attention from many researchers in the past decades (see \cite{QV99, Toselli:DDM:2005, Mathew:DDM:2008, Dolean15} and the proceedings of annual conferences on DD methods). The main idea is to decompose the domain of calculation into (overlapping or non-overlapping) subdomains with smaller sizes and then solve the subdomain problems in parallel with some transmission conditions enforced on the interfaces between the subdomains. In this work, we consider parallel Schwarz type DD with overlapping subdomains. {The parallel Schwarz algorithm} and its sequential version, namely the alternating Schwarz algorithm, were first proposed by Lions \cite{LionsI, LionsII} for stationary problems and can be extended to evolution problems {straightfowardly by first applying} time discretization to the problem and then {performing} Schwarz iteration at each time step level (consequently, the same time step size is used on the whole domain). A discrete version of the parallel Schwarz algorithm is called the additive Schwarz algorithm, which has been studied for parabolic problems in \cite{Kuznetsov90, Cai1991}. It is well known that the convergence of this type of algorithm is linear and directly dependent on the overlap sizes. Based on the idea of waveform relaxation, a new class of DD methods for parabolic problems, namely the space-time DD or overlapping Schwarz waveform relaxation method, has been introduced and studied in \cite{GanderStuart98, Gander99, GanderZhao02, Giladi2002}. Unlike the traditional approach, one decomposes the domain in both space and time and solves time-dependent problems in each subdomain at each iteration.  This approach, also called the  ``global-in-time'' method, enables the use of different time steps in different subdomains, which can be very important in some applications where the time scales in various subdomains are very different. Moreover, for short time intervals, it is shown that the algorithm converges at a super-linear rate. Hence, one could take advantage of this property by using time windows for long-term computations. 

Recently, some fast ETD algorithms, which are based on compact representation of the spatial operators and the use of linear splitting techniques to achieve further numerical stabilization, have been successfully applied to numerical simulation of grain coarsening phenomena in material science in  \cite{Ju2015-I,Ju2015-II}. A localized compact ETD algorithm based overlapping DD (i.e., perform the ETD locally in each subdomain in parallel and then pass the data of overlapping regions to the respective neighboring subdomains for time stepping) was first used in \cite{ZhangEtAl2016} for extreme-scale phase field simulations of three-dimensional coarsening dynamics in the supercomputer, and the results showed excellent parallel scalability of the method. 
{Note that the parallelism of this approach is  domain-based, which is completely different from the parallel adaptive-Krylov exponential solver proposed in \cite{Loffeld2014implementation}.}
As far as we know, neither convergence analysis nor error estimate has been theoretically studied for DD-based localized ETD methods.   In addition, it is noteworthy that unlike most existing numerical DD methods for time-dependent problems, the multidomain localized ETD problem is not algebraically equivalent to the corresponding monodomain ETD problem.
In this paper, we study localized ETD methods with parallel Schwarz algorithms for the diffusion equation discretized in space by the cell-centered finite difference method. Using either first order ETD (ETD1) or second order ETD (ETD2) approximations, a fully discrete multidomain problem is formulated whose solution is proved to converge to the exact semi-discrete (in space) solution. 
In order to solve such a multidomain problem in practice, we propose two iterative DD methods: the first method is based on Schwarz iteration applied at each time step and involves solving stationary problems in the subdomains at each iteration, while the second method is based on the Schwarz waveform relaxation algorithm in which time-dependent problems are solved in the subdomains at each iteration. We then derive a rigorous analysis indicating that the iterative solutions converge to the discrete multidomain solution at the same linear rate as the parallel Schwarz algorithm. The analysis is for one-dimensional problems and mainly based on the maximum principle. Note that explicit representations of convergence rates can only be determined for such a low dimensional case. By using the techniques in \cite{GanderZhao02} (see Remark~\ref{rmk:higherdimensions}), similar convergence results can be obtained for higher dimensional problems. 

The rest of the paper is organized as follows: in Section~\ref{sec:cont.prob}, the model problem and the parallel Schwarz method for a decomposition into two overlapping subdomains are introduced. For completeness, we recall linear and super-linear convergence results of the Schwarz waveform relaxation methods presented in \cite{GanderStuart98, Gander99, GanderZhao02, Giladi2002}. In Section~\ref{sec:letds}, we first derive fully discrete  multidomain problems using the cell-centered finite difference approximations in space and the localized ETD approximations in time,  then present formulations of different DD-based Schwarz iterative algorithms for {solving} the multidomain problem. Convergence analysis is given in Section~\ref{sec:Convergence} to show that the iterative solutions converge to the multidomain localized ETD solutions and further converge to the exact semi-discrete solution along the time step size refinement.  Numerical experiments in 1D and 2D are carried out to investigate convergence behavior of {the proposed}  algorithms and to compare their performance in Section~\ref{sec:NumRe}. Some conclusions are finally {drawn}  in Section~\ref{cons}.

\section{The model problem and parallel Schwarz waveform relaxation method} \label{sec:cont.prob}

Consider the following  time-dependent one-dimensional (in space) diffusion equation with Dirichlet boundary conditions: 
\begin{equation}\label{model}
\left \{ \begin{array}{ll} 
\frac{\partial u}{\partial t} = \nu\frac{\partial^{2} u}{\partial x^{2}} + f(x,t), &\quad 0<x<L, \; 0<t\leq T, \vspace{4pt}\\
u(0,t) =\psi_{1}(t), \quad u(L,t) =\psi_{2}(t), &\quad 0<t\leq T,\vspace{3pt}\\
u(x,0)=u_{0}(x), &\quad 0 \leq x \leq L,
\end{array} \right .
\end{equation}
where $\nu$ is a positive constant diffusion coefficient. Assume the data is sufficiently smooth so that there exists a classical solution $u \in C^{1}(0,T; C^{2}(\Omega))$.   

Let us decompose the domain $\Omega=[0,L] \times [0,T]$ into two overlapping subdomains: $\Omega_{1}=[0,\beta L] \times [0,T]$ and $\Omega_{2}=[\alpha L, L] \times [0,T]$ with $0<\alpha<\beta <1$. Extensions to many more subdomains can be done straightforwardly (see \cite{GanderZhao02} and Section~\ref{sec:NumRe}). 
\begin{figure}[!htbp]
\vspace{-0.3cm}
\begin{center}
\includegraphics[scale=0.6]{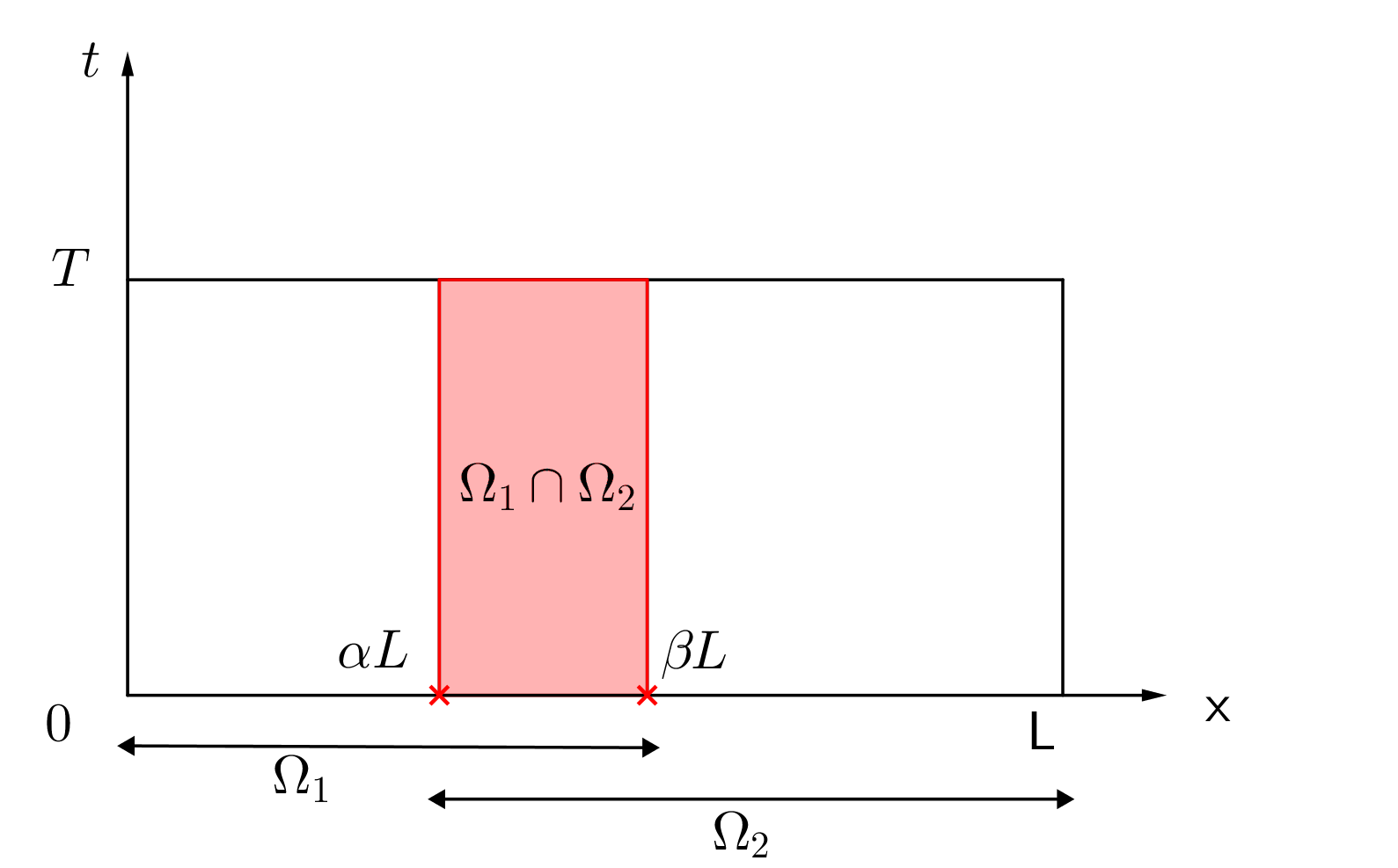} 
\end{center} 
\caption{A decomposition into two overlapping subdomains.} 
\end{figure}

\textbf{A multidomain problem} equivalent to \eqref{model} consists of solving in the subdomains the following problems:
\begin{equation}\label{modelsub1}
\left \{ \begin{array}{ll} 
\frac{\partial u_{1}}{\partial t} =  \nu\frac{\partial^{2} u_{1}}{\partial x^{2}} + f(x,t), &\quad 0<x<\beta L, \; 0<t\leq T, \vspace{4pt}\\
u_{1}(0,t) =\psi_{1}(t), &\quad 0<t\leq T,\\
u_{1}(x,0)=u_{0}(x), &\quad 0 \leq x \leq \beta L,
\end{array} \right . 
\end{equation}
and 
\begin{equation}\label{modelsub2}
\left \{ \begin{array}{ll} 
\frac{\partial u_{2}}{\partial t} =  \nu\frac{\partial^{2} u_{2}}{\partial x^{2}} + f(x,t), &\quad \alpha L <x< L, \; 0<t\leq T, \vspace{4pt}\\
u_{2}(L,t) =\psi_{2}(t), &\quad 0<t\leq T,\\
u_{2}(x,0)=u_{0}(x), &\quad \alpha L \leq x \leq L, 
\end{array} \right . 
\end{equation}
together with the transmission conditions on the interfaces of the subdomains:
\begin{equation} \label{modelTCs}
\left \{ \begin{array}{ll} u_{1}(\alpha L, t)  =u_{2} (\alpha L,t), \\
 u_{1}(\beta L, t)  =u_{2} (\beta L,t), \\
\end{array} \right . \quad 0 <t \leq T.
\end{equation}
This multidomain problem can be solved iteratively using a Schwarz-type iteration as in the elliptic case \cite{LionsI}, namely the parallel Schwarz waveform relaxation algorithm, which involves at each iteration $k=1,2, \hdots, $ the solution of 
\begin{equation}\label{Schwarz1}
\left \{ \begin{array}{ll} 
\frac{\partial u_{1}^{(k+1)}}{\partial t} =  \nu\frac{\partial^{2} u_{1}^{(k+1)}}{\partial x^{2}} + f(x,t), &\quad 0<x<\beta L, \; 0<t\leq T, \vspace{4pt}\\
u_{1}^{(k+1)}(0,t) =\psi_{1}(t), &\quad 0<t\leq T,\\
u_{1}^{(k+1)}(x,0)=u_{0}(x), &\quad 0 \leq x \leq \beta L, \vspace{4pt}\\
u_{1}^{(k+1)}(\beta L, t)  =u_{2}^{(k)} (\beta L,t), &\quad 0<t\leq T,
\end{array} \right .  \vspace{-0.2cm}
\end{equation}
and \vspace{-0.2cm}
\begin{equation}\label{Schwarz2}
\left \{ \begin{array}{ll} \frac{\partial u_{2}^{(k+1)}}{\partial t} =  \nu\frac{\partial^{2} u_{2}^{(k+1)}}{\partial x^{2}} + f(x,t), &\quad \alpha L <x< L, \; 0<t\leq T, \vspace{4pt}\\
u_{2}^{(k+1)}(L,t) =\psi_{2}(t), &\quad 0<t\leq T,\\
u_{2}^{(k+1)}(x,0)=u_{0}(x), &\quad \alpha L \leq x \leq L,  \vspace{4pt}\\
u_{2}^{(k+1)}(\alpha L, t)  =u_{1}^{(k)} (\alpha L,t), &\quad 0<t\leq T,
\end{array} \right . 
\end{equation}
where $u_{1}^{(0)} (\alpha L,t)$ and $u_{2}^{(0)} (\beta L,t)$ are given initial guess. The convergence of the Schwarz algorithm \eqref{Schwarz1}-\eqref{Schwarz2} is guaranteed by the following theorem \cite{GanderStuart98}. \vspace{5pt}

\begin{theorem} \label{thrm:conv.cont}
The Schwarz iteration \eqref{Schwarz1}-\eqref{Schwarz2} converges in $L^{\infty}([0,T], L^{\infty}([0,L]))$ to the solution $(u_{1}, u_{2})$ of \eqref{modelsub1}-\eqref{modelTCs} at a linear rate: 
\begin{align*}
\| u_{1}^{(2k+1)} - u_{1} \|_{\infty, T} \; &\leq  \left (\kappa(\alpha, \beta) \right )^{k}\;  \vert u_{2}^{(0)}(\beta L, \cdot) -u_{2}(\beta L, \cdot) \vert_{T},\vspace{4pt} \\
\| u_{2}^{(2k+1)} - u_{2}  \|_{\infty, T} \; &\leq  \left (\kappa(\alpha, \beta)\right )^{k} \; \vert u_{1}^{(0)}(\alpha L,\cdot)-u_{1}(\alpha L, \cdot) \vert_{T}, 
\end{align*}
where $\vert u \vert_{T} = \max_{0\leq t \leq T} \vert u(x,t) \vert$,  $\| u\|_{\infty, T} = \max_{x\in \Omega, 0\leq t \leq T} \vert u(x,t)\vert$ and 
$$0<\kappa(\alpha, \beta):=\frac{\alpha(1-\beta)}{\beta (1-\alpha)}<1.$$
\end{theorem}

The convergence rate is similar to that of the stationary case~\cite{LionsI} and depends on the  size of the overlap between the two subdomains.  Moreover,  for short time intervals, the convergence rate could be super-linear (see \cite{Gander99, GanderZhao02, GiladiReport}): 
\begin{theorem} \label{thrm:superlinear}
For bounded time intervals, the sequence of iterates $(u_{1}^{(k)}, u_{2}^{(k)})$ in \eqref{Schwarz1}-\eqref{Schwarz2} converges super-linearly: \vspace{-0.2cm}
\begin{align*}
&\max_{i=1,2} \| u_{i}^{(k)} - u_{i}\|_{\infty, t} \hspace{0.2cm}\\
 &\quad \leq \; {\mathsf{erfc}} \left (\frac{k(\beta-\alpha)}{2\sqrt{\nu t}}\right ) \max \bigg \{\vert u_{1}^{(0)}(\alpha L,\cdot)-u_{1}(\alpha L, \cdot) \vert_{t}, \vert u_{2}^{(0)}(\beta L, \cdot) -u_{2}(\beta L, \cdot) \vert_{t} \bigg \} ,
\end{align*}
for any $0\leq t \leq T < \infty$. 
\end{theorem}
Here ${\mathsf{erfc}}(x)$ is the complementary error function satisfying
$$ \lim_{x \rightarrow \infty}{\mathsf{erfc}}(x) = \lim_{x \rightarrow \infty} \frac{2}{\sqrt{\pi}} \int_{x}^{\infty} \text{e}^{-t^{2}}\, dt = 0. 
$$
Thus the smaller the time $t$, the faster the convergence. 

\section{Localized ETD algorithms based on overlapping domain decomposition} \label{sec:letds}
Let us consider a discretization in space using the cell-centered finite difference scheme with a uniform grid of size $h = \sfrac{L}{(N+1)}$. We then obtain the following linear system of ODEs for the discrete monodomain problem \eqref{model}:
\begin{equation} \label{ODEs}
\left \{ \begin{array}{l} \frac {\partial \bu }{\partial t}  = \bA_{(N)} \bu (t) + {\bF(t)}, \quad 0< t< T, \vspace{4pt}\\
\bu (0)  = \bu_{0},
\end{array} \right.
\end{equation}
where $\bu(t)=(u_{1},u_{2},\cdots,u_{N})^{\top}$,
\begin{align*}
 \bA_{(N)} = \frac{\nu}{h^2} \left [ \begin{array}{rrrrr} -2 & 1 & 0&\cdots & 0 \\
1 &-2 & 1&\cdots &  0\\
0&1&-2&\ddots&\vdots\\
 \vdots&\vdots & \ddots & \ddots & 1 \\
0 & \cdots &0  & 1&-2
\end{array} \right ], \quad
\bF (t) = \left (  \begin{array}{c} f(h, t) + \frac{\nu}{h^2}\psi_{1}(t) \vspace{3pt}\\
f(2h,t) \\
\vdots \\
f((N-1)h,t)\\ 
f(Nh,t) + \frac{\nu}{h^2}\psi_{2}(t)
\end{array} \right ). \vspace{-0.2cm}
\end{align*}
{Unless otherwise specified, we write $\bF(t):=\bF(f(t), \psi_{1}(t), \psi_{2}(t))$ for the sake of simplicity}. 

\subsection{Monodomain ETD schemes} 
For the time discretization, consider a partition of the time interval $[0,T]$: $0= t_{0} < t_{1} < \hdots < t_{M} = T, \; \text{with a step size} \; \Delta t~=~\sfrac{T}{M}.$
The exact (in time) solution to \eqref{ODEs} at each time level is given by the variation-of-constants formula: \vspace{-0.2cm}
\begin{equation*} \label{MonoETD}
\bu (t_{m+1}) = \text{e}^{\Delta t \bA} \bu (t_{m}) + \int_{0}^{\Delta t} \text{e}^{(\Delta t-s) \bA} \bF \left (t_{m}+s\right ) \, ds,
\end{equation*}
for $m=0,1, \hdots, M-1$. \vspace{4pt}

The first-order (monodomain) ETD scheme (also known as the exponential Euler method) based on \eqref{MonoETD} for solving the model problem \eqref{model}, denoted by ETD1, is obtained by assuming that $\bF(t)$ is constant over $(t_{m}, t_{m+1}]$: \vspace{-0.2cm}
\begin{equation} \label{IETD1Mono}
\begin{array}{rcl}
\bu_{m+1} &=& \text{e}^{\Delta t \bA} \bu_{m} + \int_{0}^{\Delta t} \text{e}^{(\Delta t-s) \bA}  \bF\left (t_{m+1} \right ) \, ds \vspace{4pt}\\
&=& \text{e}^{\Delta t \bA} \bu_{m} + \bA^{-1} \left (\text{e}^{\Delta t \bA } - \bI \right ) \bF\left (t_{m+1}\right ). 
\end{array}
\end{equation}

The second-order (monodomain) ETD scheme, ETD2, is obtained by approximating $\bF(t)$ on each time interval $[t_{m}, t_{m+1}]$ by its linear interpolation polynomial: \vspace{-0.2cm}
\begin{equation} \label{ETD2Mono}
\begin{array}{rcl}
\bu_{m+1} &= &\text{e}^{\Delta t \bA} \bu_{m} + \int_{0}^{\Delta t} \text{e}^{(\Delta t-s) \bA}  \left [ \frac{\bF\left (t_{m+1} \right ) - \bF\left (t_{m} \right )}{\Delta t} s + \bF\left (t_{m} \right ) \right ] \, ds \vspace{4pt}\\
&= & \text{e}^{\Delta t \bA} \bu_{m}  + \bA^{-1} \left (\text{e}^{\Delta t \bA} - \bI \right ) \bF\left (t_{m}\right ) \vspace{3pt}\\
&& \hspace{1.3cm}+ (\Delta t)^{-1} \bA^{-2} \left (\text{e}^{\Delta t \bA} - \bI - \Delta t \bA\right ) \left (\bF(t_{m+1}) - \bF(t_{m})\right ). 
\end{array} 
\end{equation}
For higher order exponential quadrature (for linear problems) and exponential Runge-Kutta (for semilinear problems) as well as the exponential multistep methods, we refer to \cite{hochbruckActa, Cox2002, Tokman2006} and the references therein. In this paper, we shall use either the ETD1~\eqref{IETD1Mono} or the ETD2~\eqref{ETD2Mono}.\vspace{3pt}

\subsection{Semi-discrete multidomain problem and fully discrete solutions by localized ETDs} \label{sec:discrete.prob}

For the overlapping domain decomposition approach, assume that $\alpha L = N_{\alpha} h$ and $\beta L =N_{\beta} h$ for some integers $1<N_{\alpha}<N_{\beta}<N$.
Set $N_{1} := N_{\beta}-1$, $N_{2}:=N-N_{\alpha}$ and $N_{\beta,\alpha} :=N_{\beta} - N_{\alpha}$. The semi-discrete multidomain problem corresponding to the continuous problem~\eqref{modelsub1}-\eqref{modelTCs} consists of solving the following problems:
\begin{equation} \label{semi1}
\left \{ \begin{array}{ll} \frac {\partial \bu_{1} }{\partial t}  = \bA_{1} \bu_{1} (t) + \bF_{1}(f(t), \psi_{1}(t), \bu_{2}(N_{\beta,\alpha},t)), \quad 0< t\leq T, \vspace{4pt}\\
\bu_{1} (j,0)  = \bu_{0}(j), \quad 1\leq j \leq N_{1},
\end{array} \right .  \vspace{-0.2cm}
\end{equation}
and \vspace{-0.1cm}
\begin{equation} \label{semi2}
\left \{ \begin{array}{ll} \frac {\partial \bu_{2} }{\partial t}  = \bA_{2} \bu_{2} (t) + \bF_{2}(f(t), \bu_{1}(N_{\alpha},t), \psi_{2}(t)), \quad 0< t\leq  T, \phantom{abcd}\vspace{4pt}\\
\bu_{2} (j,0)  = \bu_{0}(j+N_{\alpha}), \quad 1\leq j \leq N_{2}, 
\end{array} \right .\vspace{-0.2cm}
\end{equation}
where $\bA_{1}:= {\bA_{(N_{1})} }$, $\bA_{2}:={\bA_{(N_{2})}}$ and
\begin{align*}
&\bF_{1}(f(t), \psi_{1}(t),\bu_{2}(N_{\beta,\alpha},t))= \\
& \hspace{1cm} \left ( f(h, t) +\frac{\nu}{h^2}\psi_{1}(t), \; f(2h, t), \hdots,  f({({N_\beta}-1)}h, t)+\frac{\nu}{h^2} \bu_{2}(N_{\beta,\alpha},t)\right )^{\top}, \vspace{6pt}\\
&\bF_{2}(f(t), \bu_{1}(N_{\alpha},t), \psi_{2}(t)) = \\
& \hspace{1cm} \left ( f((N_{\alpha}+1)h, t) +\frac{\nu}{h^2} \bu_{1}(N_{\alpha},t), \; f((N_{\alpha}+2)h, t), \hdots,  f(Nh, t)+\frac{\nu}{h^2} \psi_{2}(t) \right)^{\top}. 
\end{align*}

As in the monodomain problem, ETD time-stepping methods are applied. One can use ETD1 to obtain a fully discrete multidomain solution for \eqref{semi1}-\eqref{semi2} by solving the following coupled local equations defined in $\Omega_{1}$ and $\Omega_{2}$ respectively: 
\begin{equation} \label{ETD1Multi}\left\{
\begin{array}{lcl}
\bu_{1,m+1} & =& \text{e}^{\Delta t \bA_{1}} \bu_{1,m} + \bA_{1}^{-1} \left (\text{e}^{\Delta t \bA_{1} } - \bI \right ) \bF_{1,m+1}, \vspace{3pt} \\
\bu_{2,m+1} & = &\text{e}^{\Delta t \bA_{2}} \bu_{2,m} + \bA_{2}^{-1} \left (\text{e}^{\Delta t \bA_{2} } - \bI \right ) \bF_{2,m+1},
\end{array}\right.
\end{equation}
for $m=0, 1, \hdots,  M-1$, where
\begin{align*}
\bF_{1,m}&=F_{1}\left (f(t_{m}), \psi_{1}(t_{m}), \bu_{2,m}(N_{\beta,\alpha})\right ), \vspace{3pt}\\
 \bF_{2,m}&=F_{2}\left (f(t_{m}), \bu_{1,m}(N_{\alpha}), \psi_{2}(t_{m})\right ).
\end{align*} 

Alternatively, one can also use ETD2 to obtain a fully discrete multidomain solution for  \eqref{semi1}-\eqref{semi2} by solving the following coupled local equations defined in $\Omega_{1}$ and $\Omega_{2}$ respectively: 
\begin{equation} \label{ETD2Multi}\left\{
\begin{array}{lcl}
\widetilde{\bu}_{1,m+1} &=& \text{e}^{\Delta t \bA_{1}} \bu_{m} + \bA_{1}^{-1} \left (\text{e}^{\Delta t \bA_{1}} - \bI \right ) \bF_{1,m}, \vspace{3pt}\\ 
\bu_{1,m+1} &= & \widetilde{\bu}_{1,m+1} +(\Delta t)^{-1}\bA_{1}^{-2} \left (\text{e}^{\Delta t \bA_{1}} - \bI - \Delta t \bA_{1}\right ) \left (\bF_{1,m+1} - \bF_{1,m} \right ), \vspace{4pt}\\
\widetilde{\bu}_{2,m+1} &= &\text{e}^{\Delta t \bA_{2}} \bu_{m} + \bA_{2}^{-1} \left (\text{e}^{\Delta t \bA_{2}} - \bI \right ) \bF_{2,m}, \vspace{3pt}\\ 
\bu_{2,m+1} &= & \widetilde{\bu}_{2,m+1} + (\Delta t)^{-1} \bA_{2}^{-2} \left (\text{e}^{\Delta t \bA_{2}} - \bI - \Delta t \bA_{2}\right ) \left (\bF_{2,m+1}- \bF_{2,m}\right ),
\end{array} \right.
\end{equation}
for $m=0, 1, \hdots,  M-1$. We specially remark that the localized ETD methods do not give exactly the same solutions  as those by 
the corresponding monodomain  ETD methods.
Convergence of the localized ETD1 \eqref{ETD1Multi} or the localized ETD2 \eqref{ETD2Multi} solutions to the exact semi-discrete solution \eqref{semi1}-\eqref{semi2} will be proved in Section~\ref{sec:Convergence}. 

\subsection{Schwarz iteration-based overlapping domain decomposition algorithms} \label{sec:DDalgorithms}

In order to practically compute the localized ETD  solutions for the multidomain system \eqref{ETD1Multi} or \eqref{ETD2Multi}, one need to decouple the systems in subdomains by using iterative algorithms. A straightforward extension from the classical parallel Schwarz method for elliptic problems is to perform Schwarz iteration at each time step $t_{m}$ and enforce the transmission conditions on the interfaces $\left \{x=\alpha L\right \}$ and $\left \{x=\beta L\right \}$ at~$t_{m}$. Another approach is to use global-in-time domain decomposition as presented in Section~\ref{sec:cont.prob} for continuous problems, in which time-dependent problems are solved in the subdomains and information is exchanged over the space-time interfaces $\left \{x=\alpha L \cup x = \beta L \right \} \times (0,T)$. For each method, we derive formulations using either ETD1 or ETD2  as the time marching scheme. 
%
%
\subsubsection{Method 1: Iterative, localized ETD algorithms}
For each $0 \leq m \leq M-1$, assume that $\bu_{1,m}$ and $\bu_{2,m}$  are given, we shall find the solutions at time $t_{m+1}$ by applying (parallel) Schwarz iteration. Next we construct two algorithms corresponding to the use of the  ETD1 and the ETD2 schemes for the time integration. 

\vspace{4pt}
\paragraph{\textbf{Iterative, localized  ETD1 algorithm}}
With a given  initial guess of $\bu_{1,m+1}^{(0)}(N_{\alpha})$ and $\bu_{2,m+1}^{(0)}(N_{\beta, \alpha})$, 
we compute  the subdomain solutions $\bu_{1,m+1}^{(k+1)}$ and $\bu_{2,m+1}^{(k+1)}$  by: for $k=0,1,\cdots,$
\begin{equation} \label{M2-IETD1}\left\{
\begin{array}{rl} 
 \hspace{-0.2cm}\bu_{1,m+1}^{(k+1)}\, = &\text{e}^{\Delta t \bA_{1}} \bu_{1,m}  + \bA_{1}^{-1} \left (\text{e}^{\Delta t \bA_{1}} - \bI \right ) \bF_{1}\left (f(t_{m+1}), \psi_{1}(t_{m+1}), \bu_{2,m+1}^{(k)}(N_{\beta,\alpha})\right ),\vspace{4pt}\\
 \hspace{-0.2cm}\bu_{2,m+1}^{(k+1)}\, = &\text{e}^{\Delta t \bA_{2} } \bu_{2,m}  +\bA_{2}^{-1} \left (\text{e}^{\Delta t \bA_{2} } - \bI \right ) \bF_{2}\left (f(t_{m+1}), \bu_{1,m+1}^{(k)}(N_{\alpha}), \psi_{2}(t_{m+1})\right ). 
\end{array}\right.
\end{equation}
The iteration is stopped when 
\begin{equation} \label{M1-stop} 
\frac{\vert \bu_{1,m+1}^{(k+1)}(N_{\alpha}) - \bu_{1,m+1}^{(k)}(N_{\alpha}) \vert}{\vert \bu_{1,m+1}^{(0)}(N_{\alpha})\vert } < \varepsilon \text{\,\, and \,\,}
\frac{\vert \bu_{2,m+1}^{(k+1)}(N_{\beta,\alpha}) - \bu_{2,m+1}^{(k)}(N_{\beta,\alpha})\vert }{\vert \bu_{2,m+1}^{(0)}(N_{\beta, \alpha})\vert}   < \varepsilon
\end{equation} 
for a given tolerance $\varepsilon$, then {it moves} to the next time step.

\vspace{4pt}
\paragraph{\textbf{Iterative, localized ETD2 algorithm}}
To find the solution at $t_{m+1}$, we first compute $\widetilde{\bu}_{1,m+1}$ and $\widetilde{\bu}_{2,m+1}$ from the known values of $\bu_{1,m}$ and $\bu_{2,m}$ as follows:
\begin{equation*} \label{M1-tilde}
\begin{array}{l} \widetilde{\bu}_{1,m+1} = \text{e}^{\Delta t \bA_{1}} \bu_{1,m}+ \bA_{1}^{-1} \left (\text{e}^{\Delta t \bA_{1}} - \bI \right ) \bF_{1}\left (f(t_{m}), \varphi_{1}(t_{m}), \bu_{2,m}(N_{\beta,\alpha})\right ),  \vspace{3pt}\\
\widetilde{\bu}_{2,m+1} = \text{e}^{\Delta t \bA_{2} } \bu_{2,m} + \bA_{2}^{-1} \left (\text{e}^{\Delta t \bA_{2} } - \bI \right ) \bF_{2}\left (f(t_{m}), \bu_{1,m}(N_{\alpha}), \varphi_{2}(t_{m})\right ).
\end{array}
\end{equation*}
Then we set
$$\bu_{1,m+1}^{(0)}(N_{\alpha}) = \widetilde{\bu}_{1,m+1}(N_{\alpha}) , \; \; \bu_{2,m+1}^{(0)}(N_{\beta,\alpha}) = \widetilde{\bu}_{2,m+1}(N_{\beta,\alpha}).$$
With this initial guess, we can start the iteration as: for $k=0,1,\cdots,$
\begin{equation} \label{M2-ETD2}\left\{
\begin{array}{l}
 \hspace{-0.1cm}\bu_{1,m+1}^{(k+1)}  =  \widetilde{\bu}_{1,m+1} + (\Delta t)^{-1} \bA_{1}^{-2} \left (\text{e}^{\Delta t \bA_{1}} - \bI - \Delta t \bA_{1}\right )   \vspace{4pt}\\
 \hspace{0.4cm} \cdot \Big[ \bF_{1}\left (f(t_{m+1}), \psi_{1}(t_{m+1}), \bu_{2,m+1}^{(k)}(N_{\beta,\alpha})\right ) -  \bF_{1}\left (f(t_{m}), \psi_{1}(t_{m}), \bu_{2,m}(N_{\beta,\alpha})\right ) \Big].  \vspace{4pt}\\
 \hspace{-0.2cm} \bu_{2,m+1}^{(k+1)} =  \widetilde{\bu}_{2,m+1} + (\Delta t)^{-1} \bA_{2}^{-2} \left (\text{e}^{\Delta t \bA_{2} } - \bI - \Delta t \bA_{2} \right ) \hspace{5cm} \vspace{4pt} \\
 \hspace{0.3cm}  \cdot \Big[\bF_{2}\left (f(t_{m+1}), \bu_{1,m+1}^{(k)}(N_{\alpha}), \psi_{2}(t_{m+1})\right ) - \bF_{2}\left (f(t_{m}), \bu_{1,m}(N_{\alpha}),  \psi_{2}(t_{m})\right ) \Big].
\end{array} \right.
\end{equation}
When it converges (i.e. the stopping criterion \eqref{M1-stop} is satisfied), we move to the next time step. 
%
%
\subsubsection{Method 2: Global-in-time, iterative, localized ETD algorithms}
Differently from Method~1, we can solve time-dependent problems at each iteration as a more general approach. For a given initial guess of $\left (\bu_{1,m}^{(0)}(N_{\alpha})\right )_{1\leq m \leq M}$ and $\left (\bu_{2,m}^{(0)}(N_{\beta,\alpha})\right )_{1\leq m \leq M}$, we shall compute, at the $(k+1)$-iteration, the solution $\bu_{1, m}^{(k+1)}$ and $\bu_{2, m}^{(k+1)}$ over all time steps $1\leq m\leq M$. \vspace{4pt}
\paragraph{\textbf{Global-in-time, iterative, localized ETD1 algorithm}}
Using the ETD1 scheme, we compute the approximate solution in each subdomain over all time steps $m$  in parallel:  for $k=0,1,\cdots,$ and $0\leq m \leq M-1$, 
\begin{equation} \label{M3-IETD1}\left\{
\begin{array}{ll}
\bu_{1,m+1}^{(k+1)} \;=& \text{e}^{\Delta t \bA_{1}} \bu_{1,m}^{(k+1)}+\bA_{1}^{-1} \left (\text{e}^{\Delta t \bA_{1}} - \bI \right ) \bF_{1}\left (f(t_{m+1}), \psi_{1}(t_{m+1}), \bu_{2,m+1}^{(k)}(N_{\beta,\alpha})\right ), \vspace{4pt}\\
\bu_{2,m+1}^{(k+1)}\; =& \text{e}^{\Delta t \bA_{2} } \bu_{2,m}^{(k+1)} +\bA_{2}^{-1} \left (\text{e}^{\Delta t \bA_{2} } - \bI \right ) \bF_{2}\left (f(t_{m+1}), \bu_{1,m+1}^{(k)}(N_{\alpha}), \psi_{2}(t_{m+1})\right ), 
\end{array}\right.
\end{equation}
We stop the iteration when the following conditions are satisfied:
\begin{equation}
\frac{\vert \bu_{1, \cdot}^{(k+1)}(N_{\alpha}) - \bu_{1, \cdot}^{(k)}(N_{\alpha})\vert_{T}}{\vert \bu_{1, \cdot}^{(0)}(N_{\alpha})\vert_{T}} < \varepsilon \text{\quad and \quad}
\frac{\vert \bu_{2,\cdot}^{(k+1)}(N_{\beta,\alpha}) - \bu_{2, \cdot}^{(k)}(N_{\beta,\alpha})\vert_{T}}{\vert \bu_{2, \cdot}^{(0)}(N_{\beta, \alpha})\vert_{T}} < \varepsilon.
\end{equation}
\begin{remark}
As time-dependent problems are solved in the subdomains, one may use different time grids in the subdomain and enforce the transmission conditions over nonconforming time grids by using $L^{2}$ projections \cite{GHN2003, GJaphet09}. This possibility can be very important and useful for applications in which the time scales vary by several orders of magnitude between the subdomains. 
\end{remark} \vspace{4pt}
\paragraph{\textbf{Global-in-time, iterative, localized ETD2 algorithm}}
The second order scheme can be derived similarly, in particular, we solve in parallel the following subdomain problems: for $k=0,1,\cdots,$\vspace{3pt}\\
$\bullet$ In subdomain $\Omega_{1}$: first compute
\begin{equation*} 
\begin{array}{l}
\widetilde{\bu}_{1,m+1}^{(k+1)} = \text{e}^{\Delta t \bA_{1}} \bu_{1,m}^{(k+1)} +\bA_{1}^{-1} \left (\text{e}^{\Delta t \bA_{1}} - \bI \right ) \bF_{1}\left (f(t_{m}), \psi_{1}(t_{m}), \bu_{2,m}^{(k)}(N_{\beta,\alpha})\right ),
\end{array} 
\end{equation*}
then update  \vspace{-0.2cm}
\begin{equation} \label{M3-ETD2-u1}
\begin{array}{ll}
\bu_{1,m+1}^{(k+1)} \;=&  \widetilde{\bu}_{1,m+1}^{(k+1)}+ (\Delta t)^{-1} \bA_{1}^{-2} \left (\text{e}^{\Delta t \bA_{1}} - \bI - \Delta t \bA_{1}\right ) \vspace{4pt}\\
& \hspace{-1.5cm}  \cdot \Big[ \bF_{1}\left (f(t_{m+1}), \psi_{1}(t_{m+1}), \bu_{2,m+1}^{(k)}(N_{\beta,\alpha})\right ) - \bF_{1}\left (f(t_{m}), \psi_{1}(t_{m}), \bu_{2,m}^{(k)}(N_{\beta,\alpha})\right ) \Big], \\
& \hspace{7.5cm} \,\; 0\leq m \leq M-1. 
\end{array} \vspace{-0.2cm}
\end{equation}
$\bullet$ In subdomain $\Omega_{2}$: 
\begin{equation*} 
\begin{array}{l}
\widetilde{\bu}_{2,m+1}^{(k+1)}= \text{e}^{\Delta t \bA_{2} } \bu_{2,m}^{(k+1)} +\bA_{2}^{-1} \left (\text{e}^{\Delta t \bA_{2} } - \bI \right ) \bF_{2}\left (f(t_{m}), \bu_{1,m}^{(k)}(N_{\alpha}), \psi_{2}(t_{m})\right ), 
\end{array} \vspace{-0.2cm}
\end{equation*}
then update \vspace{-0.2cm}
\begin{equation} \label{M3-ETD2-u2}
\begin{array}{ll}
\bu_{2,m+1}^{(k+1)}\; =&  \widetilde{\bu}_{2,m+1}^{(k+1)} + (\Delta t)^{-1} \bA_{2}^{-2} \left (\text{e}^{\Delta t \bA_{2} } - \bI - \Delta t \bA_{2} \right ) \vspace{4pt}\\
&\hspace{-1.5cm}  \cdot \Big[ \bF_{2}\left (f(t_{m+1}), \bu_{1,m+1}^{(k)}(N_{\alpha}), \psi_{2}(t_{m+1})\right ) -\bF_{2}\left (f(t_{m}), \bu_{1,m}^{(k)}(N_{\alpha}),  \psi_{2}(t_{m})\right ) \Big],  \\
& \hspace{7cm} \, \; 0\leq m \leq M-1. 
\end{array} \vspace{-0.1cm}
\end{equation}
Note that, for any $k$, $\bu_{1,0}^{(k)}(N_{\alpha}) = \bu_{0}(N_\alpha)$ and $\bu_{2,0}^{(k)}(N_{\beta,\alpha})=\bu_{0}(N_{\beta})$. 

\section{Convergence analysis} \label{sec:Convergence} 

We will demonstrate the convergence of the localized ETD1 or ETD2 solution $\left (\bu_{1,m}, \bu_{2, m}\right )$ to the exact semi-discrete solution $\left(\bu_{1}, \bu_{2}\right ) $ as $\Delta t \rightarrow 0$, and the convergence of the iterative solution $\left (\bu^{(k)}_{1,m}, \bu^{(k)}_{2, m}\right )$ to the corresponding localized ETD solution as $k \rightarrow \infty$. These results guarantee that the iterative solution of both methods converges to the exact solution of the model problem. The proofs are mainly based on the maximum principle of the ETD schemes and some techniques similar to those used in \cite{GanderStuart98}. 
We shall define the following discrete infinity norms: 
\begin{align*}
&  \| \bu_{m} \|_{\infty} = \max_{1\leq j \leq N} \vert u_{m}(j)\vert, \quad \vert \bu(j) \vert_{T} = \max_{1\leq m \leq M} \vert u_{m} (j) \vert, & \vspace{4pt}\\
& \hspace{1cm} \text{and} \; \; \| \bu \|_{\infty, T} = \max_{1\leq j \leq N}  \max_{1\leq m \leq M} \vert u_{m}(j) \vert,\vspace{-0.2cm}
 \end{align*}
for any {$\bu=(u_{m}(j))_{1\leq j \leq N, \, 1\leq m\leq M}$}.

\subsection{Preliminary results}  
We first present some useful results. \vspace{4pt}
%
%
\begin{lemma} \label{lmm:DMP} (Discrete Nonnegativity Property) Assume that $\bu_{m}, \, 1 \leq m~\leq~M,$ is the solution to the following problem: \vspace{-0.2cm}
\begin{equation} \label{discretesol}
\bu_{m+1}= \text{e}^{\Delta t\bA} \bu_{m}+ \int_{0}^{\Delta t} \text{e}^{(\Delta t-s) \bA} \bF(t_{m}+s) \;ds,  \quad 0 \leq m \leq M-1, 
\end{equation}
with $\bu (0) = \bu_{0}$ and $\bF (t_{m})=(\psi_{1}(t_{m}), 0, \hdots, 0, \psi_{2}(t_{m}))^{\top}$. If $\psi_{1}(t)$ and $\psi_{2}(t)$ are non-negative on $[0,T]$ and $u_{0}(j) \geq 0, \, \forall\; 1 \leq j \leq n$ then  \vspace{-0.2cm}
\begin{equation*} 
\bu_{m}\geq {\pmb{0}}, \quad  1 \leq m \leq M. \vspace{-0.2cm}
\end{equation*} 
\end{lemma}
%
%
\begin{proof}
 At the first time level $t_{1}= \Delta t$, we have: \vspace{-0.2cm} 
\begin{equation} \label{ETDsol}
 \bu_{1}= \text{e}^{\Delta t \bA} \bu_{0} +  \int_{0}^{\Delta t} \text{e}^{(\Delta t-s) \bA} \bF(s) ds, \; \; t_{1}=\Delta t. \vspace{-0.2cm}
\end{equation}
The matrix $\text{e}^{t \bA}, \, t \geq 0$ has nonnegative entries since $\bA = -2{\frac{\nu}{h^{2}}} \bI + \bM$ ($\bI$ is the identity matrix and $\bM$ contains only nonnegative entries) and  \vspace{-0.2cm}
$$ \text{e}^{t\bA} = \text{e}^{-2t{\frac{\nu}{h^{2}}}\bI } \text{e}^{t\bM}= \text{e}^{-2t{\frac{\nu}{h^{2}}}}\sum_{j=0}^{\infty} \frac{t^{j}M^{j}}{j!} \; \geq 0.  
$$
Using this and \eqref{ETDsol}, we conclude that $\bu_{1} \geq 0$ given that $\bu_{0} \geq 0$ and $\bF(s) \geq~0$ for $ 0 \leq s \leq \Delta t$. By induction, the proof is completed.  
\end{proof} 
\vspace{10pt}

%
%
\begin{lemma} \label{lmm:estimate} (Discrete maximum principle)
Assume that $\bu_{m}, \, 1 \leq m \leq M$ solves the discrete diffusion equation \eqref{discretesol} with $\bF(t_{m})=\left (\frac{\nu}{h^2} \psi_{1}(t_{m}), 0, \hdots, 0, \frac{\nu}{h^2} \psi_{2}(t_{m})\right )^{\top}$ and $\bu_{0}=0$.
Then $\bu$ satisfies the following inequality: 
\begin{equation*}
\vert \bu_{m} (j) \vert \leq \frac{N+1-j}{N+1} \vert \psi_{1}\vert_{T} + \frac{j}{N+1} \vert \psi_{2}\vert_{T}, \quad 1 \leq j \leq N, \; 1 \leq m \leq M. \vspace{-0.1cm}
\end{equation*} 
\end{lemma} 
\begin{proof}
Consider $\widetilde{\bu}$ satisfying \vspace{-0.2cm}
\begin{equation} \label{uhat}
\widetilde{\bu}_{m+1} = \text{e}^{\Delta t \bA} \widetilde{\bu}_{m} +  \int_{0}^{\Delta t} \text{e}^{(\Delta t -s) \bA} \widetilde{\bF} ds , \; \; m=0, \hdots, M-1, \vspace{-0.2cm}
\end{equation}
with \vspace{-0.2cm}
$$
\widetilde{\bu}_{0}(j) = \frac{N+1-j}{N+1} \vert \psi_{1}\vert_{T} + \frac{j}{N+1} \vert \psi_{2}\vert_{T}, \; \; \widetilde{\bF} = \left (\frac{\nu}{h^2} \vert \psi_{1}\vert_{T}, 0, \hdots, 0, \frac{\nu}{h^2} \vert \psi_{2}\vert_{T}\right )^{\top}. 
$$
We recall the following properties of the matrix $\bA$ of the cell-centered finite difference scheme: let $\bv = (1, 2, \hdots, j, \hdots, N)^{{\top}}$ and  $\widehat{\pmb{v}}= (N, N-1, \hdots, N+1-j, \hdots, 1)^{{\top}}$, then \vspace{-0.2cm}
\begin{equation*}
\bA \bv = \left (0, 0, \hdots, 0, -\nu\frac{(N+1)}{h^2}\right )^{{\top}}, \; \; \text{and} \; \; 
\bA \widehat{\pmb{v}} = \left (-\nu\frac{(N+1)}{h^2}, 0, \hdots, 0\right )^{\top}. 
\end{equation*}
Using these equations, we find that
\begin{equation*} \label{uhat0}
 \bA \widetilde{\bu}_{0}+ \widetilde{\bF}  = \pmb{0}.  \vspace{-0.2cm}
\end{equation*}
Substituting this into \eqref{uhat} at $\Delta t$ yields
\begin{align*}
\widetilde{\bu} (\Delta t) = \text{e}^{\Delta t \bA} \widetilde{\bu}_{0} +   \bA^{-1} \left (\text{e}^{\Delta t\bA} - \bI \right ) \left (-\bA {\widetilde{\bu}_{0}}\right )  = \widetilde{\bu}_{0}.
\end{align*}
By induction, we see that the solution $\widetilde{\bu}$ does not depend on time:
\begin{equation*}
\widetilde{\bu}_{m} (j) = \frac{N+1-j}{N+1} \vert \psi_{1}\vert_{T} + \frac{j}{N+1} \vert \psi_{2}\vert_{T}, \quad \; 0 \leq m \leq M, \; j=1,2, \hdots, N. 
\end{equation*}
Define $\underline{\pmb{u}}_{m}(j) = \widetilde{\bu}_{m}(j) - \bu_{m} (j)$. Then by the discrete nonnegativity property we have that $\underline{\pmb{u}}_{m} (j) \geq 0$ for all $1 \leq j \leq N$ and $ 1\leq m \leq M$. This gives 
\begin{equation*}
\bu_{m} (j) \leq  \frac{N+1-j}{N+1} \vert \psi_{1}\vert_{T} + \frac{j}{N+1} \vert \psi_{2}\vert_{T}, \quad \; 1 \leq m \leq M, 
 \; j=1,2, \hdots, N. 
\end{equation*}
Similarly, define $\overline{\pmb{u}}_{m}(j) =  \widetilde{\bu}_{m}(j) + \bu_{m} (j)$, we have that \vspace{5pt} \\
\phantom{a} $
\bu_{m} (j) \geq - \left ( \frac{N+1-j}{N+1} \vert \psi_{1}\vert_{T} + \frac{j}{N+1} \vert \psi_{2}\vert_{T}\right ), \quad 1 \leq m \leq M, \; j=1,2, \hdots, N. $ \hspace{0.4cm}
\end{proof} \vspace{0.1cm}\\
%
%
\begin{remark}\label{rmk:LemmaETDboth}
The results in Lemmas~\ref{lmm:DMP} and \ref{lmm:estimate} obviously hold if $\bu_{m}$ in \eqref{discretesol} is approximated by either ETD1 \eqref{IETD1Mono} or ETD2 \eqref{ETD2Mono}. 
\end{remark}  \vspace{0.1cm}\\
We further present a useful corollary of Lemma~\ref{lmm:estimate}. \vspace{3pt}
\begin{corollary} \label{coro:errorbound}
Assume that $\bu_{m}$ satisfies
\begin{equation} \label{coroBoundError}
\begin{array}{ll}
\vert \bu_{m+1}(j)\vert  &\leq \left \vert \left (\text{e}^{\Delta t \bA} \bu_{m} + \int_{0}^{\Delta t} \text{e}^{(\Delta t-s) \bA} \bF(t_{m}+s) ds\right )(j) \right \vert + C, \vspace{3pt}\\
& \hspace{2.5cm} \; \forall\; j=1, \hdots, N, \;  m=0, \hdots, M-1,
\end{array}
\end{equation}
with $\bu_{0}=\pmb{0}$, $\bF(t_{m})=\left (\frac{\nu}{h^2} \psi_{1}(t_{m}), 0, \hdots, 0, \frac{\nu}{h^2} \psi_{2}(t_{m})\right )^{\top}$ and $C$ a positive constant. Then
\begin{equation} \label{boundError}
\vert \bu_{m}(j) \vert \leq \frac{N+1-j}{N+1} \vert \psi_{1}\vert_{T} + \frac{j}{N+1} \vert \psi_{2}\vert_{T} + mC, \quad  m=1, \hdots, M. 
\end{equation}
\end{corollary}
\begin{proof}
The bound \eqref{boundError} is proved by induction. For $m=1$, \eqref{boundError} holds as a consequence of Lemma~\ref{lmm:estimate}. 
Now assume that \eqref{boundError} holds for some fixed $m$. Define an auxiliary solution \vspace{-0.2cm}
$$ \widetilde{\bu}_{m}(j) = \frac{N+1-j}{N+1} \left (\vert \psi_{1}\vert_{T} + mC\right ) + \frac{j}{N+1} \left (\vert \psi_{2}\vert_{T} +mC\right ), \quad j =1, \hdots, N, \vspace{-0.2cm}
$$
which satisfies  \vspace{-0.2cm}
$$ \widetilde{\bu}_{m} - \bu_{m} \geq 0, \quad \text{and} \quad \bA \widetilde{\bu}_{m} = \widetilde{\bF}_{m},  \vspace{-0.2cm}
$$
where $\widetilde{\bF}_{m} =  \left (\frac{\nu}{h^2} \left (\vert \psi_{1}\vert_{T} + mC\right ), 0, \hdots, 0, \frac{\nu}{h^2} \left (\vert \psi_{2}\vert_{T} +mC\right )\right )^{\top}$. \vspace{4pt}\\
Denote by $\uubu_{m} $ the solution to \vspace{-0.2cm}
$$ \uubu_{m} = \text{e}^{\Delta t \bA} \bu_{m} + \int_{0}^{\Delta t} \text{e}^{(\Delta t-s) \bA} \bF(t_{m}+s) ds.
$$
Using the same argument as in the proof of Lemma~\ref{lmm:estimate} and by the discrete nonnegativity property, we have that $\widetilde{\bu}_{m}- \uubu_{m}\geq 0$ and $\widetilde{\bu}_{m}+ \uubu_{m}\geq 0$. This implies \vspace{-0.3cm}
$$ \vert \uubu_{m}(j)  \vert  \leq \frac{N+1-j}{N+1} \left (\vert \psi_{1}\vert_{T} + mC\right ) + \frac{j}{N+1} \left (\vert \psi_{2}\vert_{T} +mC\right ). \vspace{-0.1cm}
$$
Inserting the above inequality into \eqref{coroBoundError}, we obtain \vspace{-0.2cm}
$$  \vert \bu_{m+1}(j) \vert \leq \frac{N+1-j}{N+1} \vert \psi_{1}\vert_{T} + \frac{j}{N+1} \vert \psi_{2}\vert_{T} + (m+1)C. 
$$
By the principle of induction, \eqref{boundError} holds for all $m$. 
\end{proof}
%
%
\subsection{Convergence of the multidomain localized ETD solutions to the exact semidiscrete solution}

We next present a detailed proof in the case that the first-order ETD method is used with a nonzero source term and nonhomogeneous Dirichlet boundary conditions (see Theorem~\ref{thrm:LETD1}). The result is then extended to the second-order case (see Theorem~\ref{thrm:LETD2}). 

Our proof relies on the representation of the exact (in time) solution to the semi-discrete multidomain problem \eqref{semi1}-\eqref{semi2} by the variation-of-constants formula:
\begin{equation*} 
\left \{ \begin{array}{rcl}
\bu_{1}(t_{m+1}) & =& \text{e}^{\Delta t \bA_{1}} \bu_{1}(t_{m}) \\
&&+ \int_{0}^{\Delta t} \text{e}^{(\Delta t-s) \bA_{1}} \bF_{1}(f(t_{m}+s), \psi_{1}(t_{m}+s), \bu_{2}(N_{\beta,\alpha}, t_{m}+s)) ds, \vspace{3pt}\\
\bu_{2}(t_{m+1}) &= &\text{e}^{\Delta t \bA_{2}} \bu_{2}(t_{m}) \\
&&+ \int_{0}^{\Delta t} \text{e}^{(\Delta t-s) \bA_{2}} \bF_{2}(f(t_{m}+s), \bu_{1}(N_{\alpha}, t_{m}+s), \psi_{2}(t_{m}+s)) ds,
\end{array} \right . \vspace{-0.2cm}
\end{equation*}
for $m=0, \hdots, M-1$ with initial conditions as in \eqref{semi1}-\eqref{semi2}. 

Denote by $\be_{i,m} = \bu_{i}(t_{m})- \bu_{i,m},$ the error between the exact (in time) solution and the fully discrete localized ETD1 solution \eqref{ETD1Multi}, which satisfies: \vspace{-0.2cm}
\begin{align} \label{errorE11}
\be_{1,m+1} &= \;\text{e}^{\Delta t \bA_{1}} \be_{1,m} + \int_{0}^{\Delta t} \text{e}^{(\Delta t-s) \bA_{1}} \bF_{1}\left (0,0,\bu_{2}(N_{\beta,\alpha}, t_{m}+s)-\bu_{2,m+1}(N_{\beta,\alpha})\right ) ds \nonumber\\
& \hspace{-0.4cm}+\int_{0}^{\Delta t} \text{e}^{(\Delta t-s) \bA_{1}} \bF_{1}\left (f(t_{m}+s) - f(t_{m+1}), \psi_{1}(t_{m}+s) -\psi_{1}(t_{m+1}) , 0\right ) ds, \vspace{-0.2cm}
\end{align}
and  \vspace{-0.2cm}
\begin{align} \label{errorE2}
\be_{2,m+1} &=\; \text{e}^{\Delta t \bA_{2}} \be_{2,m} + \int_{0}^{\Delta t} \text{e}^{(\Delta t-s) \bA_{2}} \bF_{2}\left (0,\bu_{1}(N_{\alpha}, t_{m}+s) - \bu_{1,m_1}(N_{\alpha}),0\right ) ds \nonumber\\
& \hspace{-0.4cm}+\int_{0}^{\Delta t} \text{e}^{(\Delta t-s) \bA_{2}} \bF_{2}\left (f(t_{m}+s) - f(t_{m+1}), 0,\psi_{2}(t_{m}+s) -\psi_{2}(t_{m+1})\right ) ds,\vspace{-0.2cm}
\end{align}
for $m=0, \hdots, M-1$ with $\be_{1,0} =\be_{2,0}=\pmb{0}$. We have the following convergence result.\vspace{3pt}
%
%
%
%
\begin{theorem} \label{thrm:LETD1}
For sufficiently smooth data, the  localized ETD1 method converges as $\Delta t$ tends to $0$. More precisely, the following error bound holds: 
\begin{equation}\label{semierr1}
  \Vert \be_{1,\cdot}\Vert_{\infty, T} + \Vert \be_{2,\cdot}\Vert_{\infty, T} \leq C \Delta t,
\end{equation}
where $C$ is a constant depending on $T$, the size of overlap, the mesh size $h$, $\bu_{1}^{\prime}(N_{\alpha},t)$, $\bu_{2}^{\prime}(N_{\beta,\alpha},t)$, the source term $f$ and the boundary data. 
\end{theorem} \vspace{4pt}

\begin{proof}
From {\eqref{errorE11}}, we have that for any $0\leq m\leq M-1$:
\begin{align} \label{errorE1}
&\vert \be_{1,m+1}(j) \vert \vspace{4pt} \nonumber\\
& \hspace{0cm} \leq \left \vert \left (\text{e}^{\Delta t \bA_{1}} \be_{1,m} + \int_{0}^{\Delta t} \text{e}^{(\Delta t-s) \bA_{1}} \bF_{1}\left (0,0,\bu_{2}(N_{\beta,\alpha}, t_{m}+s)-\bu_{2,m+1}(N_{\beta,\alpha})\right ) ds\right )(j)  \right \vert \vspace{4pt}\nonumber\\
&\quad+ \left \vert \left (\int_{0}^{\Delta t} \text{e}^{(\Delta t-s) \bA_{1}} \int_{s}^{\Delta t} \bF_{1}\left (f^{\prime}(t_{m}+\tau), \psi^{\prime}_{1}(t_{m}+\tau), 0\right ) d\tau \, ds \right )(j)\right \vert. \vspace{-0.2cm}
\end{align}
Now using the fact that, {for any vector $\bu_{1} \in \mR^{N_{1}}$ and any $t \in [0,T]$: } 
\begin{equation}
\begin{array}{rcl}
 \vert (\text{e}^{t\bA_{1}} \bu_{1})(j) \vert &\leq& \Vert \text{e}^{t\bA_{1}} \bu_{1} \Vert_{\infty} 
\leq  \Vert  \text{e}^{t\bA_{1}} \Vert_{\infty} \Vert \bu_{1} \Vert_{\infty} \vspace{2pt}\\
 &\leq&   \sqrt{N_1}\Vert\text{e}^{t\bA_{1}} \Vert_{2} \Vert \bu_{1} \Vert_{\infty} \leq \sqrt{L/h}\; \|\bu_{1} \|_{\infty}, 
\end{array}
\end{equation}
for all $j =1, \hdots, N_{1}$, we can bound the last term of \eqref{errorE1} by \vspace{-0.2cm}
\begin{align*}
&\left \vert \left (\int_{0}^{\Delta t} \text{e}^{(\Delta t-s) \bA_{1}} \int_{s}^{\Delta t} \bF_{1}\left (f^{\prime}(t_{m}+\tau), \psi^{\prime}_{1}(t_{m}+\tau), 0\right ) d\tau \, ds\right )(j) \right \vert \\
&\hspace{1.2cm} \leq {(\Delta t)^{2}} \underbrace{\sqrt{L/h}\;\left (  \sup_{\substack{x \in (0,\beta L) \\ t\in (0,T)}} \vert f^{\prime}(x,t)\vert+ \frac{\nu}{h^{2}}\sup_{t\in (0,T)}\vert\psi_{1}^{\prime}(t) \vert \right )}_\text{$C_{1}$} \leq C_{1}(\Delta t)^{2}. \vspace{-0.2cm}
\end{align*}
This together with \eqref{errorE1} and Corollary~\ref{coro:errorbound} yields \vspace{-0.2cm}
\begin{equation} \label{error1a}
\begin{array}{rcl}\vert \be_{1,m+1} (j)\vert &\leq& \frac{j}{N_{1}+1} \left(\max_{0\leq l\leq M-1} \sup_{s \in [0,\Delta t]} \vert \bu_{2}(N_{\beta,\alpha}, t_{l}+s)-\bu_{2,l+1}(N_{\beta,\alpha}) \vert \right)\\
&& \hspace{0cm} + (m+1) C_{1}{(\Delta t)^{2}}.
\end{array} 
\end{equation}

Moreover, we have that  for $0\leq l\leq M-1$,
\begin{align*}
&\sup_{s \in [0,\Delta t]} \vert \bu_{2}(N_{\beta,\alpha}, t_{l}+s)-\bu_{2,l+1}(N_{\beta,\alpha}) \vert \\\
&\qquad=\sup_{s \in [0,\Delta t]} \left \vert \bu_{2}(N_{\beta,\alpha}, t_{l+1})-\bu_{2,l+1}(N_{\beta,\alpha}) - \int_{s}^{\Delta t} \bu_{2}^{\prime} (N_{\beta,\alpha}, t_{l}+\tau) \;d\tau \right  \vert\\
&\qquad \leq \vert \be_{2,l+1}(N_{\beta,\alpha})\vert + \Delta t\sup_{s \in [0,\Delta t]} \vert \bu_{2}^{\prime} (N_{\beta,\alpha}, t_{l}+s) \vert. 
\end{align*}
Inserting this into \eqref{error1a}, we deduce that 
\begin{equation}\label{error1b}
\vert \be_{1,m+1}(j)\vert  \leq \frac{j}{N_{1}+1} \left [\max_{0\leq l \leq M-1} \vert \be_{2,l+1}(N_{\beta,\alpha})\vert + \Delta t\sup_{t \in [0,T]} \vert \bu_{2}^{\prime} (N_{\beta,\alpha}, t) \vert\right ]  + C_{1}T \Delta t .
\end{equation} 
Following a same argument, one can obtain a bound for $\be_{2,m+1}$:
\begin{equation}\label{error2b}
\vert \be_{2,m+1}(j)\vert \leq \frac{N_{2}+1-j}{N_{2}+1} \left [\max_{0\leq l\leq M-1} \vert \be_{1,l+1}(N_{\alpha})\vert +\Delta t\sup_{t \in [0,T]} \vert \bu_{1}^{\prime} ( N_{\alpha}, t) \vert \right ] + C_{2}T \Delta t, \vspace{-0.2cm}
\end{equation}
where $C_{2} =  \sqrt{L/h}\;\left (\sup_{\substack{x \in (\alpha L, L) \\ t\in (0,T)}} \vert f^{\prime}(x,t)\vert+ \frac{\nu}{h^{2}}\sup_{t\in (0,T)}\vert\psi_{2}^{\prime}(t) \vert\right )$. \vspace{4pt}

Evaluate \eqref{error1b} with $j=N_{\alpha}$ and \eqref{error2b} with $j=N_{\beta,\alpha}$, then combine the two resulting inequalities (note that their right-hand sides do not depend on $m$) to obtain:  
\begin{equation}\label{twoinequal}\left\{
\begin{array}{rl}
\vert \be_{1,m+1}(N_{\alpha})\vert  & \leq\; \kappa (\alpha, \beta) \max_{0\leq l\leq M-1}  \vert \be_{1,l+1}(N_{\alpha})\vert + \widetilde{C} \Delta t,\vspace{0.2cm}\\
\vert \be_{2,m+1}(N_{\beta,\alpha})\vert  & \leq\; \kappa (\alpha, \beta) \max_{0\leq l\leq M-1} \vert \be_{2,l+1}(N_{\beta,\alpha})\vert + \widetilde{C} \Delta t,
\end{array}\right.
\end{equation}
where
\begin{align*} 
& \widetilde{C} = \sup_{t \in [0,T]} \vert \bu_{1}^{\prime} ( N_{\alpha}, t) \vert + \sup_{t \in [0,T]} \vert \bu_{2}^{\prime} (N_{\beta,\alpha}, t) \vert  + (C_{1}+C_{2})T.
\end{align*}
Note that to derive \eqref{twoinequal}, we have used the following equality: 
$$\frac{N_{\alpha}}{N_{1}+1} \left (\frac{N_{2}+1-(N_{\beta,\alpha})}{N_{2}+1}\right ) =\frac{N_{\alpha}}{N_{\beta}} \left (\frac{N+1-N_{\beta}}{N+1-N_{\alpha}}\right )= \frac{\alpha (1-\beta)}{\beta (1-\alpha)}=\kappa (\alpha, \beta). 
$$
Substituting \eqref{twoinequal} into \eqref{error1b} and \eqref{error2b}, we find that \vspace{-0.1cm}
\begin{equation} \label{prffinal}\left\{
\begin{array}{rcl}\Vert \be_{1,m+1}\Vert_{\infty} & \leq & \max_{0\leq l\leq M-1} \vert \be_{2,l+1}(N_{\beta,\alpha})\vert + \Delta t\sup_{t \in [0,T]} \vert \bu_{2}^{\prime} (N_{\beta,\alpha}, t) \vert + C_{1}T\Delta t\\
& \leq& \kappa (\alpha, \beta) \max_{0\leq l\leq M-1} \vert \be_{2,l+1}(N_{\beta,\alpha})\vert + \overline{C}_{1} \Delta t, \vspace{2pt}\\
\Vert \be_{2,m+1}\Vert_{\infty} &\leq &\max_{0\leq l\leq M-1} \vert \be_{1,l+1}(N_{\alpha})\vert + \Delta t\sup_{t \in [0,T]} \vert \bu_{1}^{\prime} ( N_{\alpha}, t) \vert +C_{2}T\Delta t \\
& \leq& \kappa (\alpha, \beta) \max_{0\leq l\leq M-1}  \vert \be_{1,l+1}(N_{\alpha})\vert + \overline{C}_{2} \Delta t,
\end{array} \right. \vspace{-0.2cm}
\end{equation}
where \vspace{-0.2cm}
$$\overline{C}_{1}=\widetilde{C}+\sup_{t \in [0,T]} \vert \bu_{2}^{\prime} (N_{\beta,\alpha}, t) \vert + C_{1}T, \quad \overline{C}_{2}=\widetilde{C}+\sup_{t \in [0,T]} \vert \bu_{1}^{\prime} ( N_{\alpha}, t) \vert +C_{2}T.$$
Since the terms on the right hand side of \eqref{prffinal} do not depend on $m$, we can deduce that \vspace{-0.1cm}
\begin{align*}
\Vert \be_{1,\cdot}\Vert_{\infty, T} & \leq \kappa (\alpha, \beta)  \Vert \be_{2,\cdot}\Vert_{\infty,T}+ \overline{C}_{1} \Delta t,\\
\Vert \be_{2,\cdot}\Vert_{\infty, T} & \leq \kappa (\alpha, \beta)  \Vert \be_{1,\cdot}\Vert_{\infty,T}+ \overline{C}_{2} \Delta t.
\end{align*}
Thus we have 
\begin{align*}
&\left (1-\kappa (\alpha, \beta) \right ) \left (\Vert \be_{1,\cdot}\Vert_{\infty, T} + \Vert \be_{2,\cdot}\Vert_{\infty, T}\right ) \leq {(\overline{C}_{1}+\overline{C}_{2})}  \Delta t,
\end{align*}
which gives us \eqref{semierr1}.
If the data is sufficiently smooth, the error will tend to zero as $\Delta t$ approaches $0$. 
\end{proof} 

The convergence of the   localized ETD2 method  can be proved using similar techniques. Denote by $\widehat{\be}_{i,m}$ the error between the exact solution to \eqref{semi1}-\eqref{semi2} and the fully discrete localized ETD2 solution \eqref{ETD2Multi}. We have the following results. \vspace{3pt}
%
%
%
%
\begin{theorem} \label{thrm:LETD2}
For sufficiently smooth data, the localized ETD2 method converges as $\Delta t$ tends to $0$:
$$  \Vert \widehat{\be}_{1,\cdot}\Vert_{\infty, T} + \Vert \widehat{\be}_{2,\cdot}\Vert_{\infty, T} \leq C (\Delta t)^{2},
$$
where $C$ is a constant depending on $T$, the  size of overlap, the mesh size $h$, $\bu_{1}^{\prime\prime}(N_{\alpha},t)$, $\bu_{2}^{\prime\prime}(N_{\beta,\alpha},t)$, the source term $f$ and the boundary data. 
\end{theorem} \vspace{4pt}

\begin{proof} We follow similar arguments as in Theorem \ref{thrm:LETD1} but skip some details. For simplicity, assume that $f=0$ and $\psi_{1}=\psi_{2}=0$, we use Taylor series twice with the remainder in integral form and write the exact solution, for instance, in $\Omega_{1}$ as follows:
\begin{align*}
\bu_{1}(t_{m+1}) 
= &\;  \text{e}^{\Delta t \bA_{1}} \bu_{1}(t_{m}) + \int_{0}^{\Delta t} \text{e}^{(\Delta t-s) \bA_{1}} \bF_{1}{(0, 0, \bu_{2}(N_{\beta,\alpha}, t_{m}))} ds \\
&+ \int_{0}^{\Delta t} \text{e}^{(\Delta t-s) \bA_{1}} \left [\frac{\bF_{1}{(0, 0, \bu_{2}(N_{\beta,\alpha}, t_{m+1}))}-\bF_{1}{(0, 0, \bu_{2}(N_{\beta,\alpha}, t_{m}))}}{\Delta t} \right ] s\, ds\\
& + \gamma_{1,m+1},
\end{align*}
for $m=0, \hdots, M-1$, where
\begin{align*} \gamma_{1,m+1} =& \int_{0}^{\Delta t} \text{e}^{(\Delta t-s) \bA_{1}} \left (\int_{0}^{\Delta t} (\Delta t-\tau)  \bF_{1}{(0, 0, \bu_{2}^{\prime \prime}(N_{\beta,\alpha}, t_{m}+\tau))} d\tau\right ) s \,  ds \\
&
+\int_{0}^{\Delta t} \text{e}^{(\Delta t-s) \bA_{1}} \int_{0}^{s} (s-\tau)  \bF_{1}{(0, 0, \bu_{2}^{\prime \prime}(N_{\beta,\alpha}, t_{m}+\tau))} \;d\tau \,  ds.
\end{align*}
The error between the exact solution and the localized, ETD2 solution~\eqref{ETD2Multi} satisfies:
\begin{align*}
\widehat{\be}_{1,m+1}=&\;\text{e}^{\Delta t \bA_{1}} \widehat{\be}_{1,m}+ \int_{0}^{\Delta t} \text{e}^{(\Delta t-s) \bA_{1}} \bF_{1}{(0, 0, \widehat{\be}_{2,m}(N_{\beta,\alpha}))} \\
&+ \int_{0}^{\Delta t} \text{e}^{(\Delta t-s) \bA_{1}} \left [\frac{\bF_{1} {(0, 0, \widehat{\be}_{2,m+1}(N_{\beta,\alpha}))} - \bF_{1}{(0, 0, \widehat{\be}_{2,m}(N_{\beta,\alpha}))}}{\Delta t}\right ] s \, ds \\
&+ \gamma_{1,m+1}. 
\end{align*}
Note that $\gamma_{1,m+1}(j), \, 1\leq j \leq N_{1},$ is bounded by $C_* (\Delta t)^{3}$ where $C_{*}$ depends on the supremum of $\bu_{2}^{\prime \prime}(N_{\beta,\alpha},t)$ for $t \in (0,T)$. Using Remark~\ref{rmk:LemmaETDboth} and Corollary~\ref{coro:errorbound}, we can  obtain a bound for $\widehat{\be}_{1,m}$ as follows:
$$ \vert \widehat{\be}_{1,m+1} (j) \vert \leq \frac{j}{N_{1}+1} \max_{0\leq l\leq M-1} \vert \widehat{\be}_{2,l+1} (N_{\beta,\alpha}) \vert + C_{1} T (\Delta t)^{2}, \quad 1 \leq j \leq N_{1},
$$
for some constant $C_1$. 
Similarly, one can derive a bound for $\widehat{\be}_{2,m}$. Following same arguments as in \eqref{twoinequal} and so on, we finally obtain 
$$ 
\left (1-\kappa (\alpha, \beta) \right ) \left (\Vert \widehat{\be}_{1,\cdot}\Vert_{\infty, T} + \Vert \widehat{\be}_{2,\cdot}\Vert_{\infty, T}\right ) \leq C T(\Delta t)^{2}, \vspace{-0.2cm}
$$
for some constant $C$ depending on the mesh size $h$, the supremums of $\bu_{1}^{\prime\prime}(N_{\alpha},t)$ and $\bu_{2}^{\prime\prime}(N_{\beta, \alpha},t)$ on $(0,T)$.  
\end{proof}

\subsection{Convergence of the Schwarz iterative solutions to the corresponding localized ETD solutions} 

We will  show in Theorem~\ref{thrm:M3.ETD1} that Method~2 converges at a similar linear rate as in the continuous problem (cf. Theorem~\ref{thrm:conv.cont}). The rate depends only on the size of overlap  but neither on the mesh size nor the time step size.  The convergence of Method~1 is obtained as a consequence of Theorem~\ref{thrm:M3.ETD1} (see Remark~\ref{rmk:convM1}). \vspace{4pt}
%
%
\begin{theorem} \label{thrm:M3.ETD1} 
The sequence of iterates $(\bu_{1}^{(k)}, \bu_{2}^{(k)})$ given by Method 2 (with ETD1 \eqref{M3-IETD1} $($or ETD2 {\eqref{M3-ETD2-u1}-\eqref{M3-ETD2-u2}}$)$ converges to the discrete solution $(\bu_{1}, \bu_{2})$ in \eqref{ETD1Multi} $($or \eqref{ETD2Multi}$)$ as $k\rightarrow \infty$:
$$ \| \bu_{1}^{(k)} - \bu_{1}\|_{\infty, T} + \| \bu_{2}^{(k)} - \bu_{2}\|_{\infty, T} \rightarrow 0, \; \; \text{as $k \rightarrow \infty$}.  
$$
In particular: 
\begin{align*}
\| \bu_{1}^{(2k+1)} - \bu_{1} \|_{\infty, T} &\leq  \left (\kappa(\alpha, \beta)\right )^{k}   \vert \bu_{2, \cdot}^{(0)}(N_{\beta,\alpha}) -\bu_{2, \cdot}(N_{\beta,\alpha}) \vert_{T}, \vspace{4pt}\\
\| \bu_{2}^{(2k+1)} - \bu_{2}  \|_{\infty, T} &\leq   \left (\kappa(\alpha, \beta)\right )^{k} \vert \bu_{1, \cdot}^{(0)}(N_{\alpha})-\bu_{1, \cdot}(N_{\alpha}) \vert_{T}. \vspace{-0.3cm}
\end{align*}\vspace{-0.2cm}
\end{theorem}

\begin{proof} 
Define the errors at each iteration: 
$$\bw_{1,m}^{(k+1)} = \bu_{1,m}^{(k+1)} - \bu_{1,m}, \quad \bw_{2,m}^{(k+1)} = \bu_{2,m}^{(k+1)} - \bu_{2,m},$$ 
that satisfy the following equations:  for $m=0,1, \hdots, M-1,$ 
\begin{enumerate}
\item[i)] if ETD1 is used:
\begin{equation*}
\begin{array}{rcl}
\bw_{1,m+1}^{(k+1)}&=& \text{e}^{\Delta t \bA_{1}} \bw_{1,m}^{(k+1)} + \int_{0}^{\Delta t} \text{e}^{(\Delta t-s) \bA_{1}} \bF_{1}(0,0,\bw^{(k)}_{2,m+1}(N_{\beta,\alpha})), \vspace{3pt} \\
\bw_{2,m+1}^{(k+1)}&=& \text{e}^{\Delta t \bA_{2}} \bw_{2,m}^{(k+1)} + \int_{0}^{\Delta t} \text{e}^{(\Delta t-s) \bA_{2}} \bF_{2}(0,\bw^{(k)}_{1,m+1}(N_{\alpha}),0), 
\end{array}
\end{equation*}
\item[ii)] if ETD2 is used:
\begin{equation*}
\begin{array}{ll}
\bw_{1,m+1}^{(k+1)}=& \text{e}^{\Delta t \bA_{1}} \bw_{1,m}^{(k+1)} \vspace{3pt}\\
& + \int_{0}^{\Delta t} \text{e}^{(\Delta t-s) \bA_{1}} \bigg [ \frac{\bF_{1}(0,0,\bw^{(k)}_{2,m+1}(N_{\beta,\alpha})) - \bF_{1}(0,0,\bw^{(k)}_{2,m}(N_{\beta,\alpha}))}{\Delta t} s \\
&\qquad\qquad\qquad\qquad + \bF_{1}(\bw^{(k)}_{2,m}(0,0,N_{\beta,\alpha})) \bigg ] \, ds, \vspace{4pt} \\
\bw_{2,m+1}^{(k+1)}=& \text{e}^{\Delta t \bA_{2}} \bw_{2,m}^{(k+1)} \vspace{3pt}\\
&  + \int_{0}^{\Delta t} \text{e}^{(\Delta t-s) \bA_{2}} \bigg [ \frac{\bF_{2}(0,\bw^{(k)}_{1,m+1}(N_{\alpha}),0) - \bF_{2}(0,\bw^{(k)}_{1,m}(N_{\alpha}),0)}{\Delta t} s \\
&\qquad\qquad\qquad\qquad + \bF_{2}(0,\bw^{(k)}_{1,m}(N_{\alpha}),0) \bigg ] \, ds.
\end{array}
\end{equation*}
\end{enumerate}
For both cases, the initial conditions are $\bw_{1,0}^{(k+1)} = \bw_{2,0}^{(k+1)}=\pmb{0}$. 
By Lemma \ref{lmm:estimate} and Remark~\ref{rmk:LemmaETDboth}, we have that \vspace{-0.2cm}
\begin{align*}
\vert \bw_{1,m}^{(k+1)}(j) \vert & \leq \frac{j}{N_{1}+1} \vert \bw_{2, \cdot}^{(k)}(N_{\beta,\alpha}) \vert_{T}, \quad1\leq j \leq N_{\beta}-1,\\
\vert \bw_{2,m}^{(k+1)}(j) \vert & \leq \frac{N_{2}+1-j}{N_{2}+1} \vert \bw_{1, \cdot}^{(k)}(N_{\alpha}) \vert_{T},  \quad 1\leq j \leq N-N_{\alpha}, 
\end{align*}
from which we deduce that (as in \cite[Lemma 2.7]{GanderStuart98} and \eqref{twoinequal})
\begin{equation*}
\begin{array}{rcl}
\vert \bw_{1, \cdot}^{(2k)}(N_{\alpha} ) \vert_{T} &\leq&  \left (\kappa(\alpha, \beta)\right )^{k} \vert \bw_{1, \cdot}^{(0)}(N_{\alpha}) \vert_{T}, \vspace{5pt}\\
\vert \bw_{2, \cdot}^{(2k)}(N_{\beta,\alpha}) \vert_{T} &\leq&  \left (\kappa(\alpha, \beta)\right )^{k} \vert \bw_{2, \cdot}^{(0)}(N_{\beta,\alpha}) \vert_{T}.
\end{array} 
\end{equation*}
Using again Lemma \ref{lmm:estimate} and these inequalities we finally obtain \vspace{-0.2cm}
\begin{align*}
\| \bw_{1}^{(2k+1)}\|_{\infty, T} &\leq \vert \bw_{2, \cdot}^{(2k)}(N_{\beta,\alpha})\vert_{T}  \leq \left (\frac{\alpha (1-\beta)}{\beta (1-\alpha)}\right )^{k} \vert \bw_{2, \cdot}^{(0)}(N_{\beta,\alpha}) \vert_{T}. 
\end{align*}
A similar result can be proved for $\bw_{2}$. 
\end{proof} \vspace{3pt}

\begin{remark} \label{rmk:convM1}
Method~1 can be regarded as Method~2 with only one time step $T=\Delta t$. Consequently, the convergence of Method~1 is straightforward from Theorem~\ref{thrm:M3.ETD1}. Moreover, according to the super-linear convergence of the continuous Schwarz waveform relaxation method for short time intervals (see Theorem~\ref{thrm:superlinear}), one would expect that the convergence rate of Method~1 would depend also on the time step size $\Delta t$. We shall verify this numerically when we study the convergence of both methods versus the time step size in the next section.
\end{remark} \vspace{3pt}

\begin{remark}  \label{rmk:higherdimensions}
For ease of understanding, we consider the one dimensional case and derive explicit formulas for the constant involved in the convergence of the fully discrete solutions and for the convergence rates of the iterative solutions. The analysis presented above can be extended to higher dimensional problems using again the maximum principle and with $\kappa(\alpha, \beta)$ being replaced by some $\mathfrak{K}(\delta)<1$ depending on the size of overlap $\delta$ (see \cite{GanderZhao02} for the case of continuous problems). We shall present numerical results for one and two dimensional examples in the following section. 
\end{remark}
%
%
%
%
\section{Numerical results} \label{sec:NumRe}

In this section we numerically study   convergence behavior and accuracy of the localized ETD  algorithms presented in Section~\ref{sec:DDalgorithms}. 
In Subsection~\ref{subsec:zero.test}, the 1D error equation (with zero solution) is considered to investigate the dependence of the convergence speed of  Schwarz 
iteration in {Method~1 (the iterative, localized ETD) and Method~2 (the global-in-time, iterative, localized ETD)} on the size of overlap, on the time step size and on the length of the time interval when the domain is decomposed into two overlapping subdomains. We also show convergence for the case with many subdomains. In Subsection~\ref{subsec:anal.test} we consider a 1D example with an analytical solution and verify the temporal accuracy of the multidomain localized ETD solutions. Finally, we present numerical results for a 2D test case in Subsection~\ref{subsec:2D}.  

\subsection{The 1D error equation for testing the convergence of Schwarz iteration} \label{subsec:zero.test}

The spatial domain $\Omega=[0,2]$ is split into two non-overlapping subdomains $\widetilde{\Omega}_{1}=[0, 1]$ and $\widetilde{\Omega}_{2}=[1, 2]$ with an interface $\Gamma=\{{x~=~1}\}$. We enlarge each $\widetilde{\Omega}_{i}$ a distance $\delta \in (0,1)$ to obtain overlapping subdomains $\Omega_{1}~=~[0, 1+\delta]$ and $\Omega_{2}=[1-\delta, 2]$. The overlap size is equal to $2 \delta$ and will be chosen to be proportional to the mesh size. In order to study the convergence behavior of the two methods, we consider the error equation with a zero solution, i.e. we solve the model problem with a zero source term, a zero initial condition and homogeneous Dirichlet boundary conditions.  We start the iteration with a random initial guess on the interfaces between subdomains. In particular, for Method~1 and at the time step $t_{m} \, (m\geq 1)$, the initial interface guess values are $\textstyle\bu_{1,m}^{(0)}(N_{\alpha})$ and $\textstyle\bu_{2,m}^{(0)}(N_{\beta, \alpha})$, while for Method~2 the initial guess consists of two vectors of size $M$, $\textstyle\bu_{1,\cdot}^{(0)}(N_{\alpha})$ and $\textstyle\bu_{2,\cdot}^{(0)}(N_{\beta, \alpha})$. All the components are chosen randomly in the interval $(0,1)$. \vspace{4pt}

At each iteration we compute the errors in $L^{\infty}(\Omega)$-norm and in $L^{\infty}(0,T, L^{\infty}(\Omega))$-norm for Method 1 and Method 2 respectively. Note that to show the error reduction, we shall normalize the errors at each iteration by the error of the first iteration. \vspace{5pt}

\paragraph{\textbf{Convergence vs. different overlap sizes}}

We fix $T = 1$, $h = 2/256 \approx 0.0078$,  $\Delta t=0.01$, and take various  $\delta \in \{ h, 2 h, 4 h, 8 h, 16 h \}$. To see the effect of the overlap size on the convergence rate, we plot the normalized errors in logarithmic scale at each Schwarz iteration for different sizes of the overlap. For Method~1, the errors for the first time level $t=\Delta t$ are shown Figure~\ref{fig:M2.overlap} (the numbers of iterations for the following time levels are usually smaller than the first level, but their convergence behavior is similar). Clearly, the larger the size of overlap, the faster the convergence. Moreover, the errors decay quite faster if one uses the  localized ETD2  instead of the localized  ETD1 (by a factor of nearly 2). 

\begin{figure}[!htbp]
\centering
\begin{minipage}{0.45 \linewidth}
\centering
\includegraphics[scale=0.24]{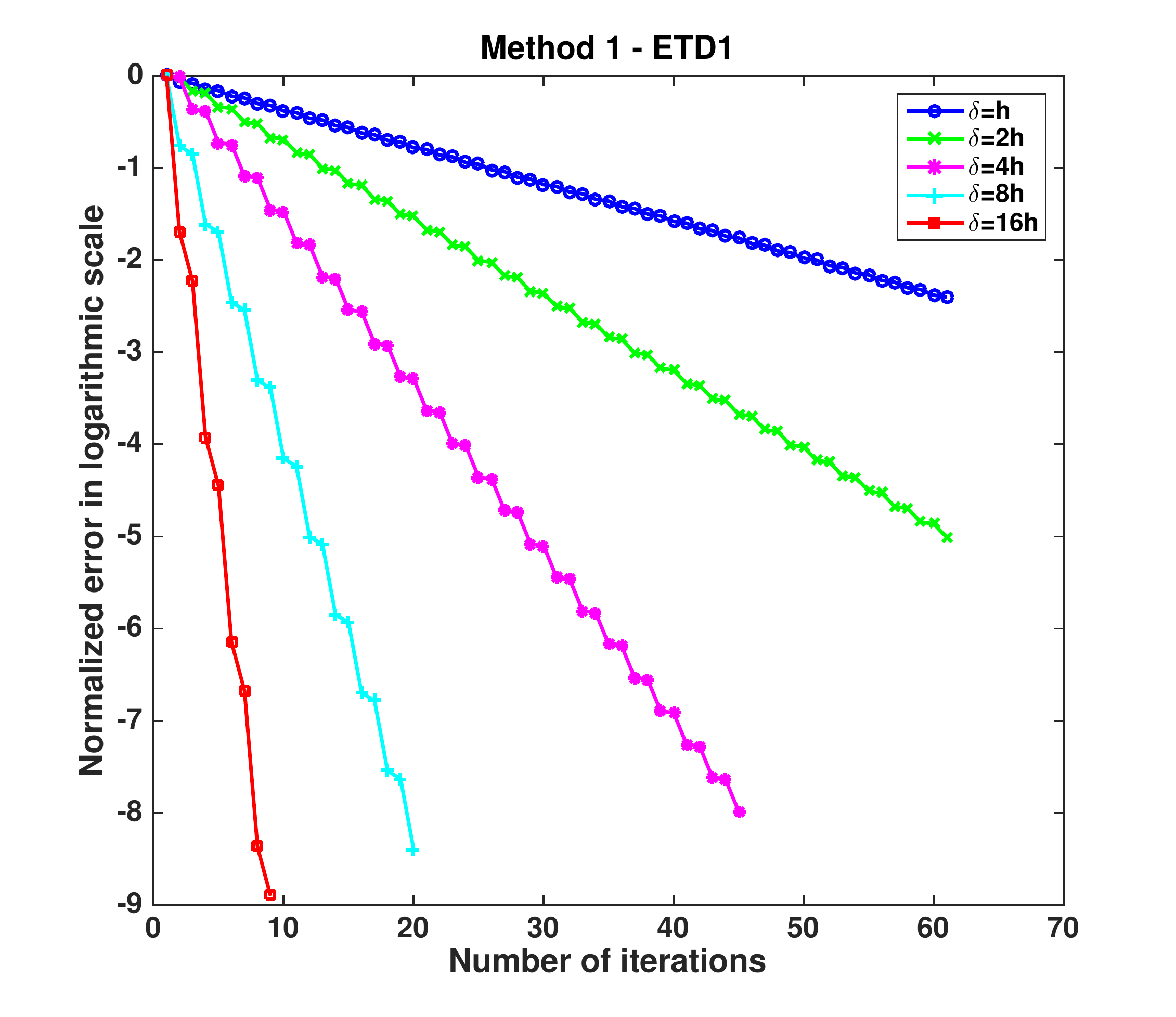}
\end{minipage} \hspace{5pt}
\begin{minipage}{0.45 \linewidth}
\centering
\includegraphics[scale=0.24]{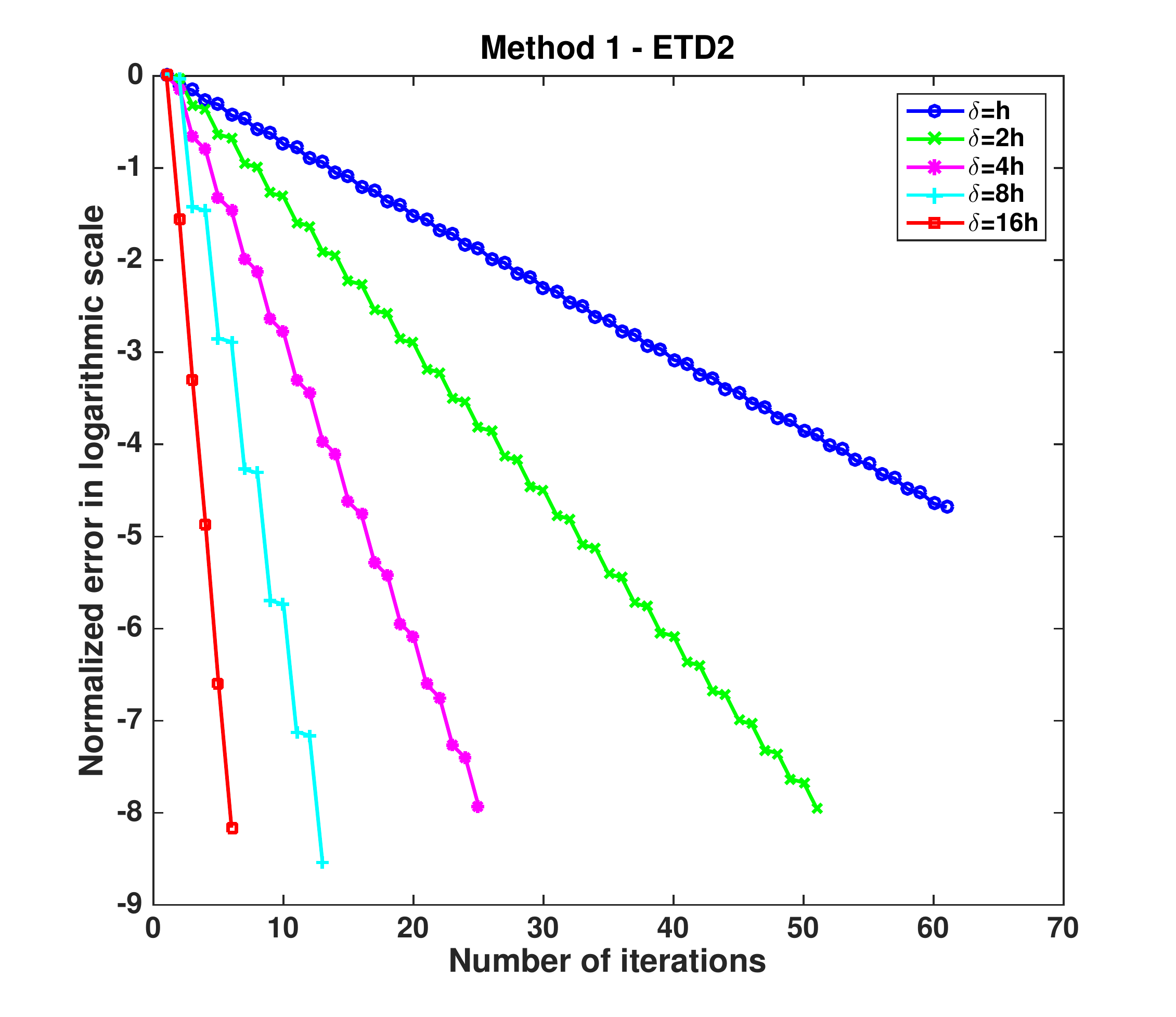}
\end{minipage}
\caption{Decay curves of the normalized $L^{\infty}(\Omega)$ errors of Method 1 at $t=\Delta t$ for different sizes of overlap, with the  ETD1 (left) or  the ETD2 (right). } \label{fig:M2.overlap} \vspace{-0.3cm}
\end{figure}

For Method~2, the errors over the whole time interval is presented in Figure~\ref{fig:M3.overlap}. The number of Schwarz iterations is for the whole time interval, not at each time level as in Method~1. We observe that the size of overlap  has a profound effect in this case. However, we do not observe a significant difference between the localized ETD1 and the localized ETD2 in terms of number of iterations required to obtain similar error reduction. 

\begin{figure}[!htbp]
\centering
\begin{minipage}{0.45 \linewidth}
\centering
 \includegraphics[scale=0.24]{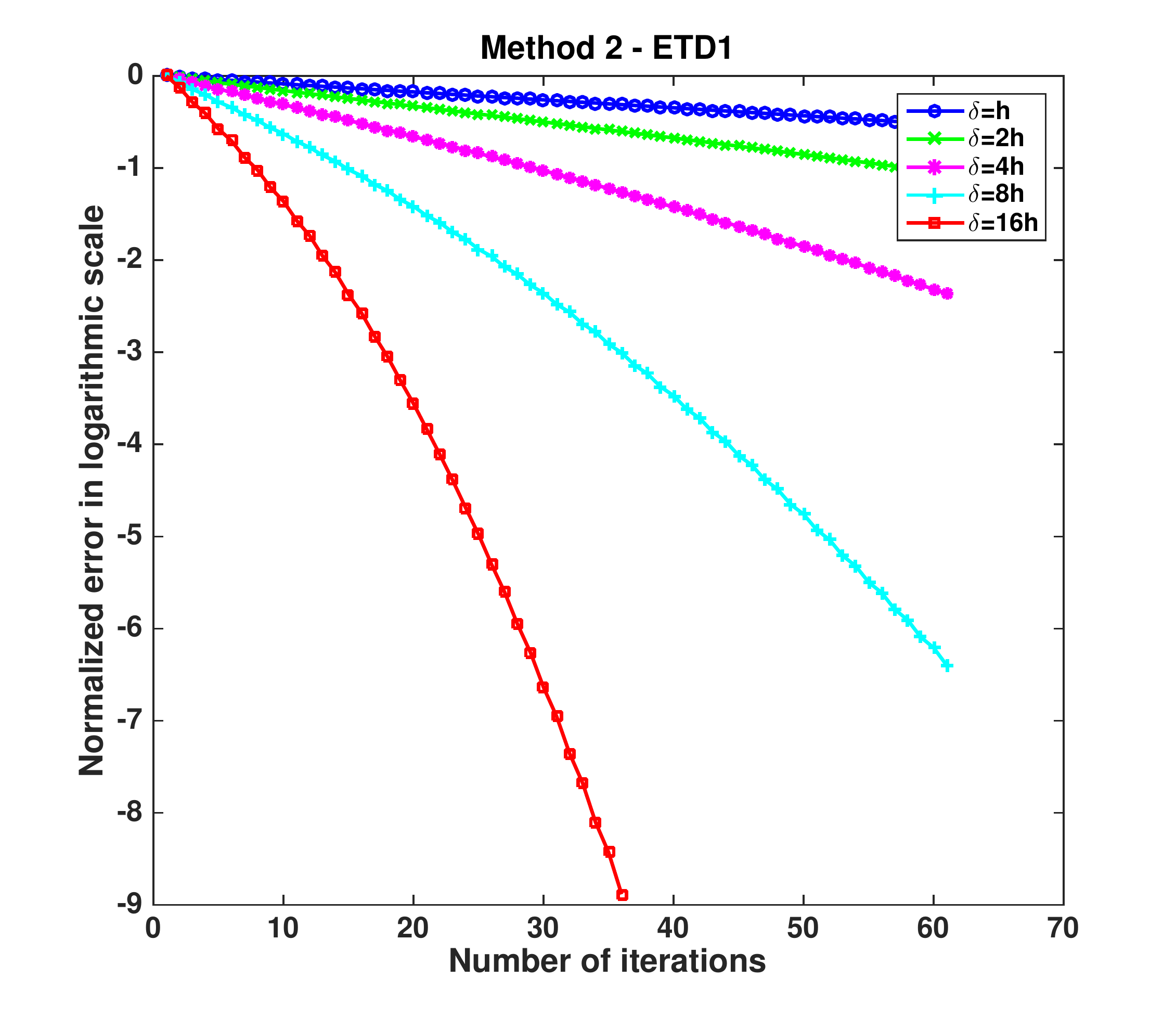}
\end{minipage} \hspace{5pt}
\begin{minipage}{0.45 \linewidth}
\centering
\includegraphics[scale=0.24]{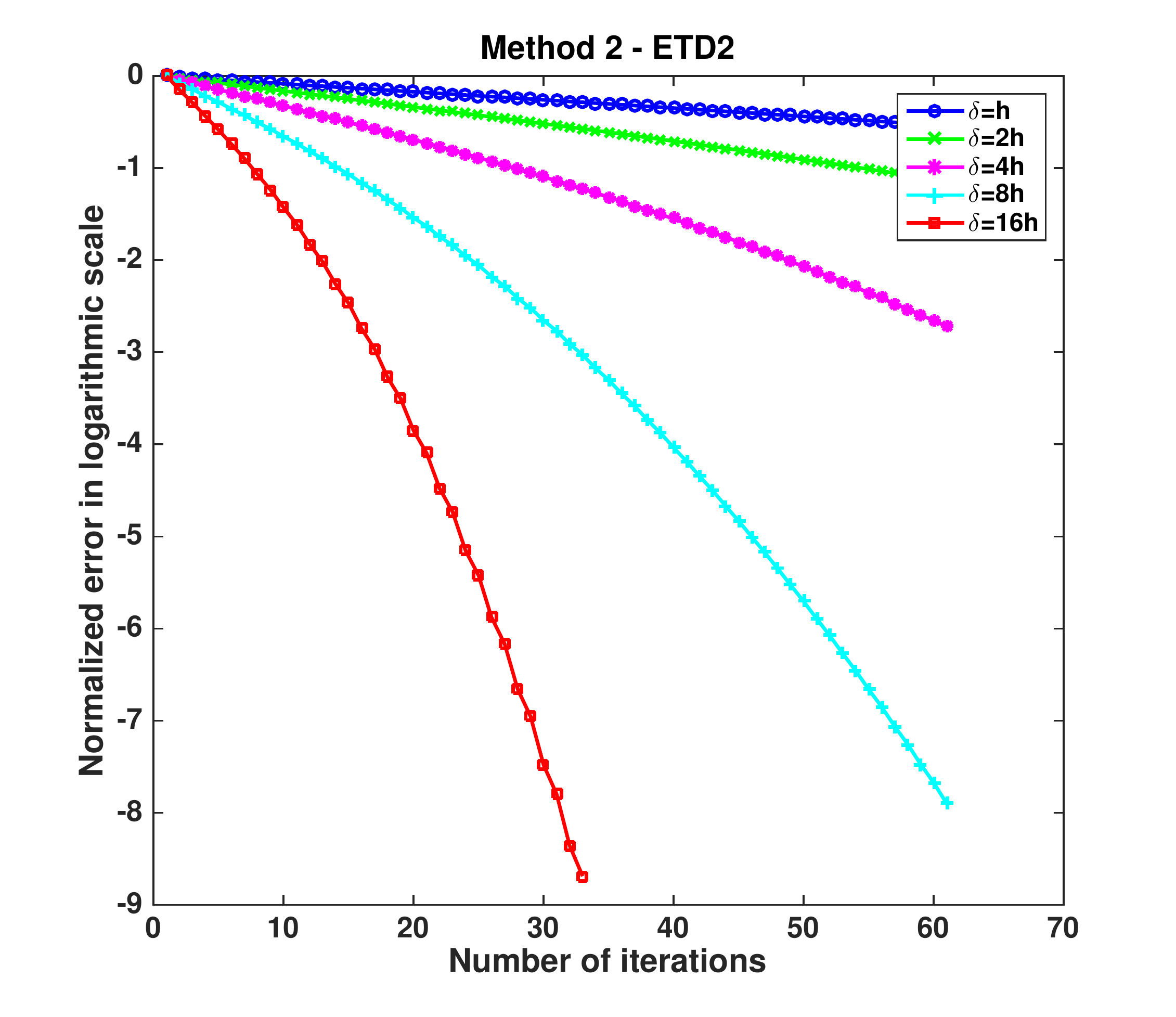}
\end{minipage}
\caption{Decay curves of the normalized $L^{\infty}(0,T, L^{\infty}(\Omega))$ errors of Method 2 over $[0,T]$ for different sizes of overlap, with the  ETD1 (left) or  the ETD2 (right). } \label{fig:M3.overlap} \vspace{-0.6cm}
\end{figure}

In Table~\ref{tab:rate}, we compare theoretical and simulated  decay rates of the {normalized} errors $\textstyle \frac{\vert \bw_{1}^{(2k)}(N_{\alpha}, \Delta t)\vert}{\vert \bw_{1}^{(0)}(N_{\alpha}, \Delta t)\vert}$ for Method 1 and $\textstyle\frac{\vert \bw_{1, \cdot}^{(2k)}(N_{\alpha})\vert_{T}}{\vert \bw_{1, \cdot}^{(0)}(N_{\alpha})\vert_{T}}$ for Method 2 {with respect to the number of iterations} for different  sizes of overlap. We see that the numerical rates of Method~2 are quite consistent with the theory, while for Method~1, the error decays at a linear rate but much faster than theoretical prediction. For evolution problems, the space domain decomposition behaves differently from the case of elliptic problems and one should take into account also the effect of the time step. The next results further confirm this effect. 

\begin{table}[!htbp]
\centering
\scriptsize
\setlength{\extrarowheight}{3pt}
\begin{tabular}{|c | c | c | c| c| c | }
\hline 
\multicolumn{2}{|c|}{Method}	& $\delta = h$	& $\delta = 2h$	& $\delta = 4h$	& $\delta = 8h$ 	\\
\hline
\hline
\multicolumn{2}{|c|}{Theoretical rate $\frac{\alpha(1-\beta)}{\beta(1-\alpha)}$} & $0.97$ & $0.94$ & $0.88$ & $0.78$  \\ \hline 
\multirow{2}{*}{Method 1} & ETD1				 & $0.91$ &  $0.83$ & $0.66$ & $0.38$ 
\\ \cline{2-6} 
											& ETD2 			& $0.84$ & $0.69$ & $0.47$ & $0.20$ 
\\ \hline
\multirow{2}{*}{Method 2} & ETD1 & $0.97$ & $0.96$ & $0.92$ & $0.80$ 
\\ \cline{2-6}
& ETD2 & $0.98$ & $0.96$ & $0.92$ & $0.76$
\\ \hline
\end{tabular}
\caption{Theoretical and simulated decay rates of the normalized errors for the two methods.} \label{tab:rate}  \vspace{-0.6cm}
\end{table}

\paragraph{\textbf{Convergence vs. different time step sizes}}

We fix the size of overlap with $\delta = 8 h$, the final time $T=1$ and take  various $\Delta t \in \{ 0.2, 0.1, 0.05, 0.025, 0.0125\}$. We show the error evolution curves for different time step sizes in Figure~\ref{fig:M2.dt} (Method 1 in which the normalized errors are computed at the first time level) and Figure~\ref{fig:M3.dt} (Method 2 where the normalized errors are computed in the whole time interval).  For Method~1, it is easy to find that the convergence is very sensitive to the time step size - the smaller the time step, the faster the rate; again, the error decays much  faster in the case of the iterative localized ETD2 than the iterative  localized ETD1 (by a factor of 3 now). 
For Method~2, however, the results show that it is quite independent of the time step, especially when the ETD2 is used. Hence, one can use large time steps without increasing significantly the number of iterations. In addition, the ETD2 always gives much smaller errors than the ETD1 using the same number of iterations. 

\begin{figure}[!htbp]
\centering
\begin{minipage}{0.45 \linewidth}
\centering
 \includegraphics[scale =0.25]{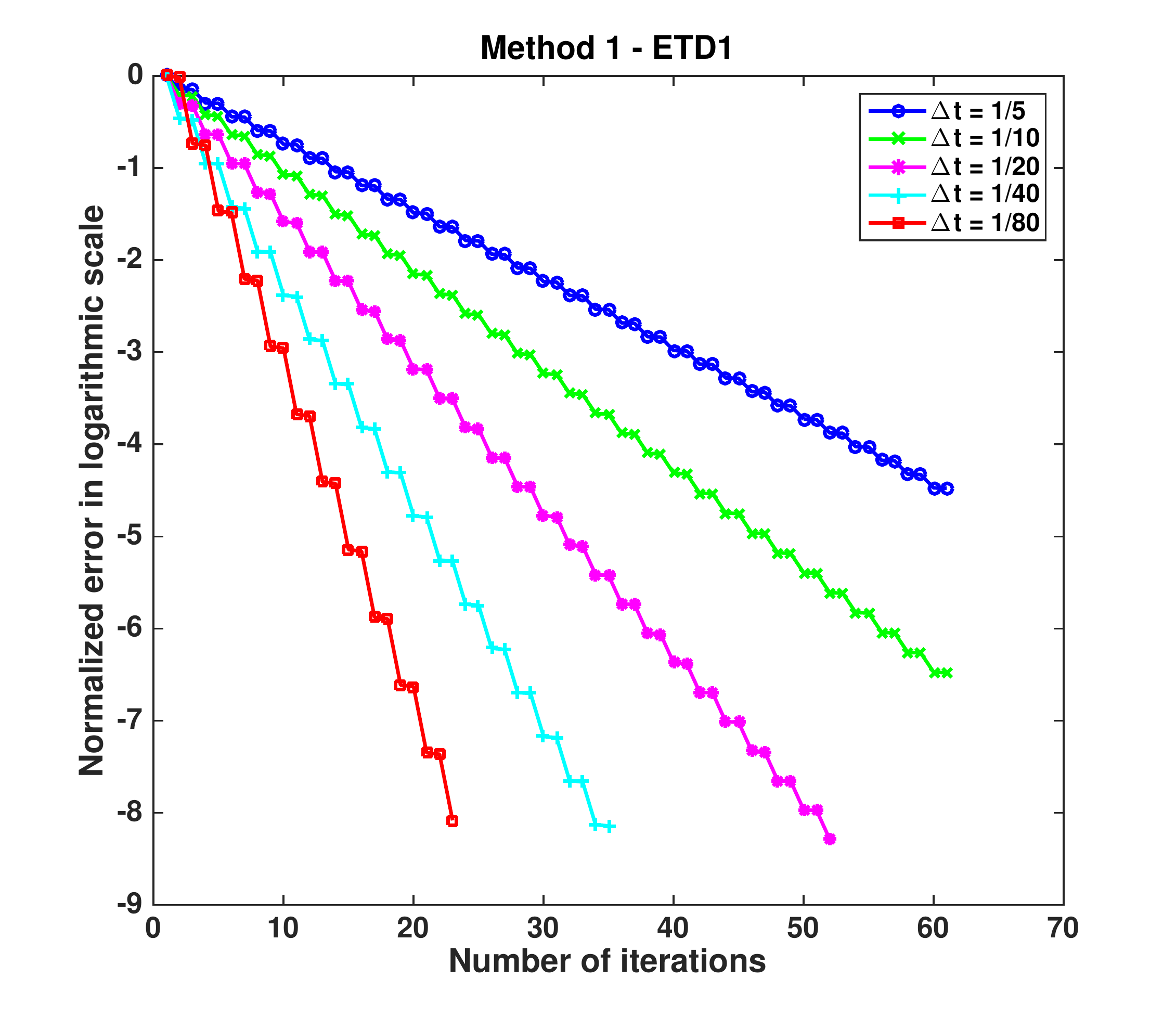}
\end{minipage} \hspace{5pt}
\begin{minipage}{0.45 \linewidth}
\centering
\includegraphics[scale=0.24]{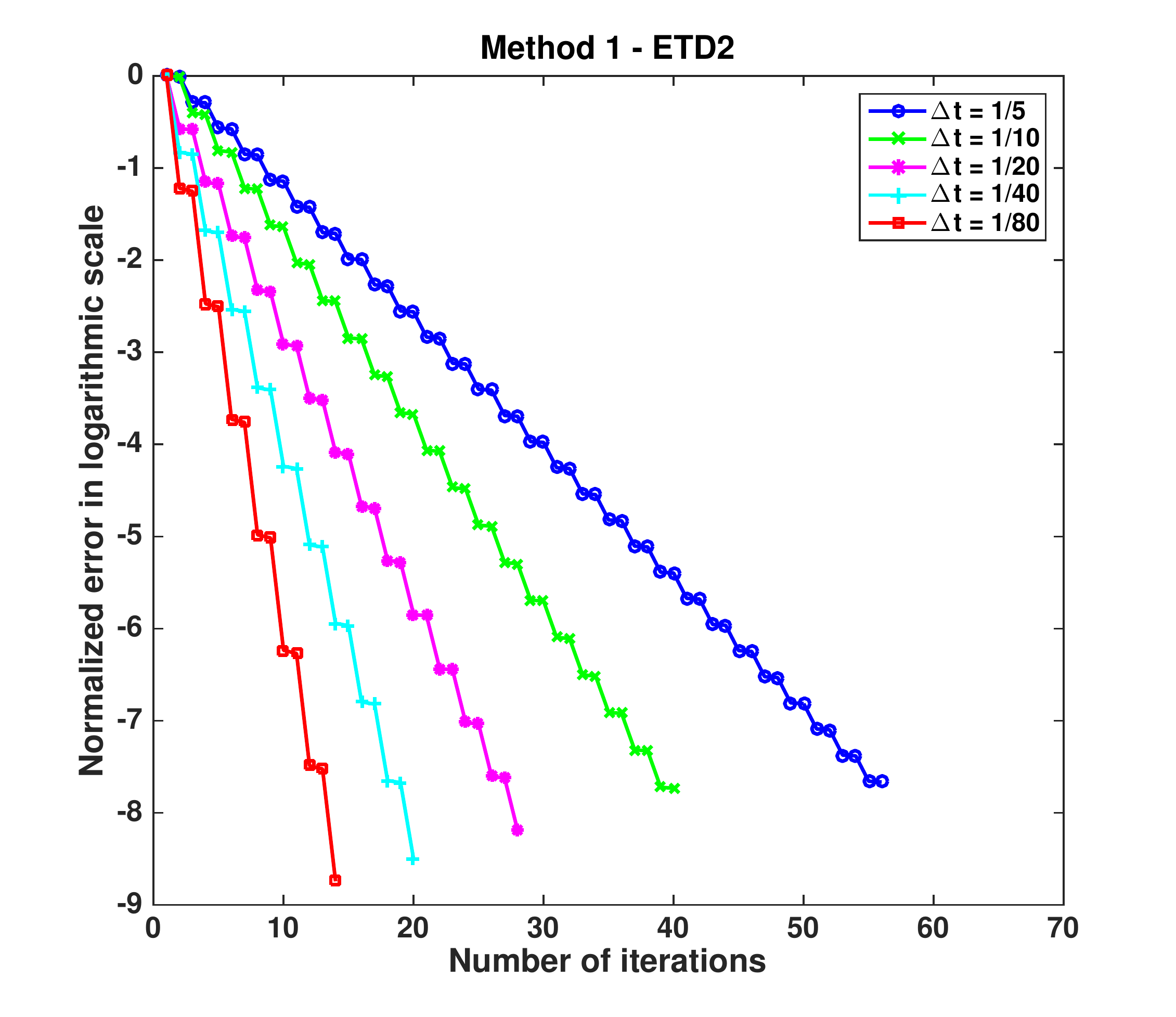}
\end{minipage}
\caption{Decay curves of the normalized $L^{\infty}(\Omega)$ errors of Method 1 at the first time level $t=\Delta t$ for different time step sizes, with the  ETD1 (left) or  the ETD2 (right). } \label{fig:M2.dt} \vspace{-0.6cm}
\end{figure}
\begin{figure}[!htbp]
\centering
\begin{minipage}{0.45 \linewidth}
\centering
 \includegraphics[scale=0.24]{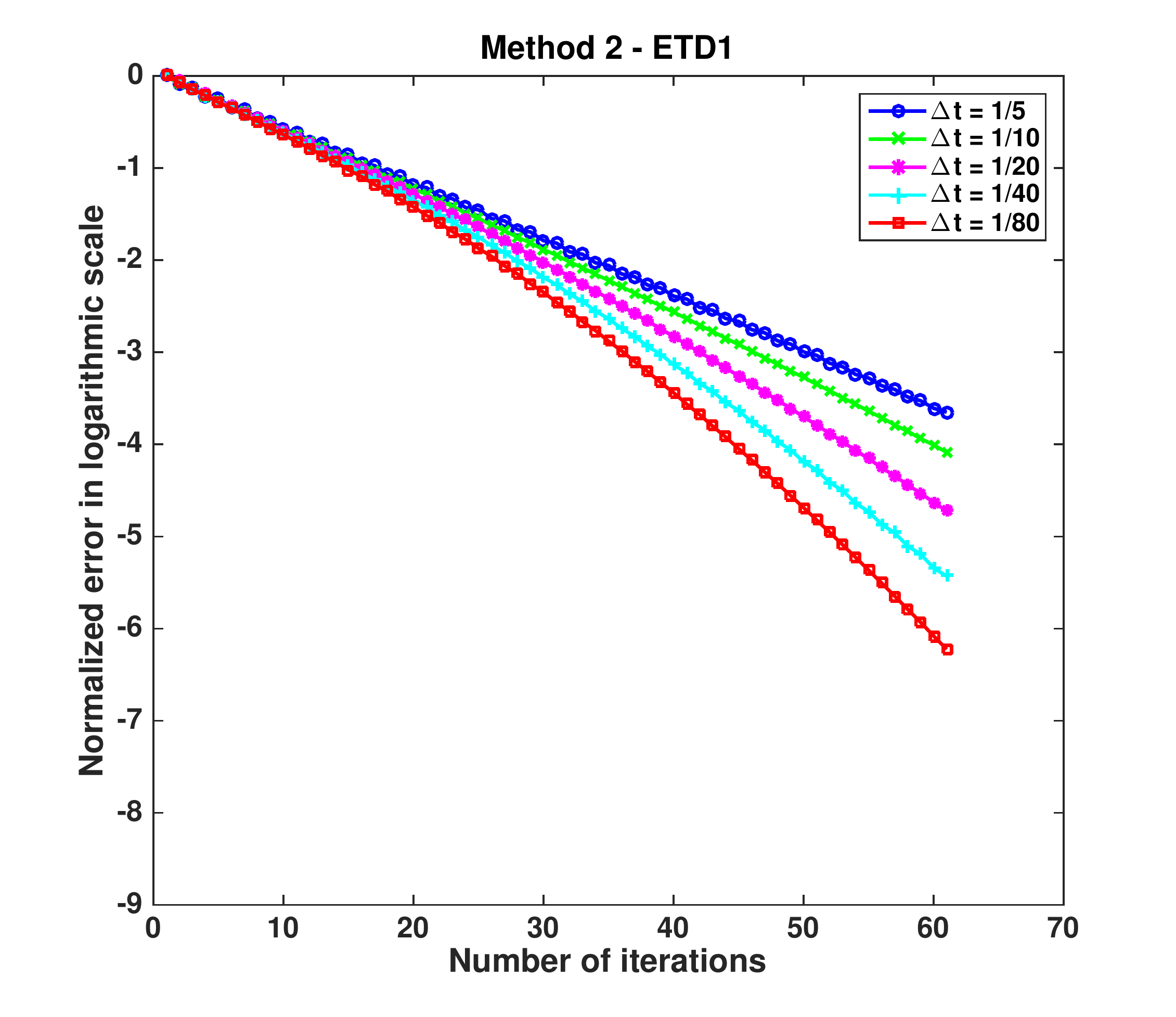}
\end{minipage} \hspace{5pt}
\begin{minipage}{0.45 \linewidth}
\centering
\includegraphics[scale=0.24]{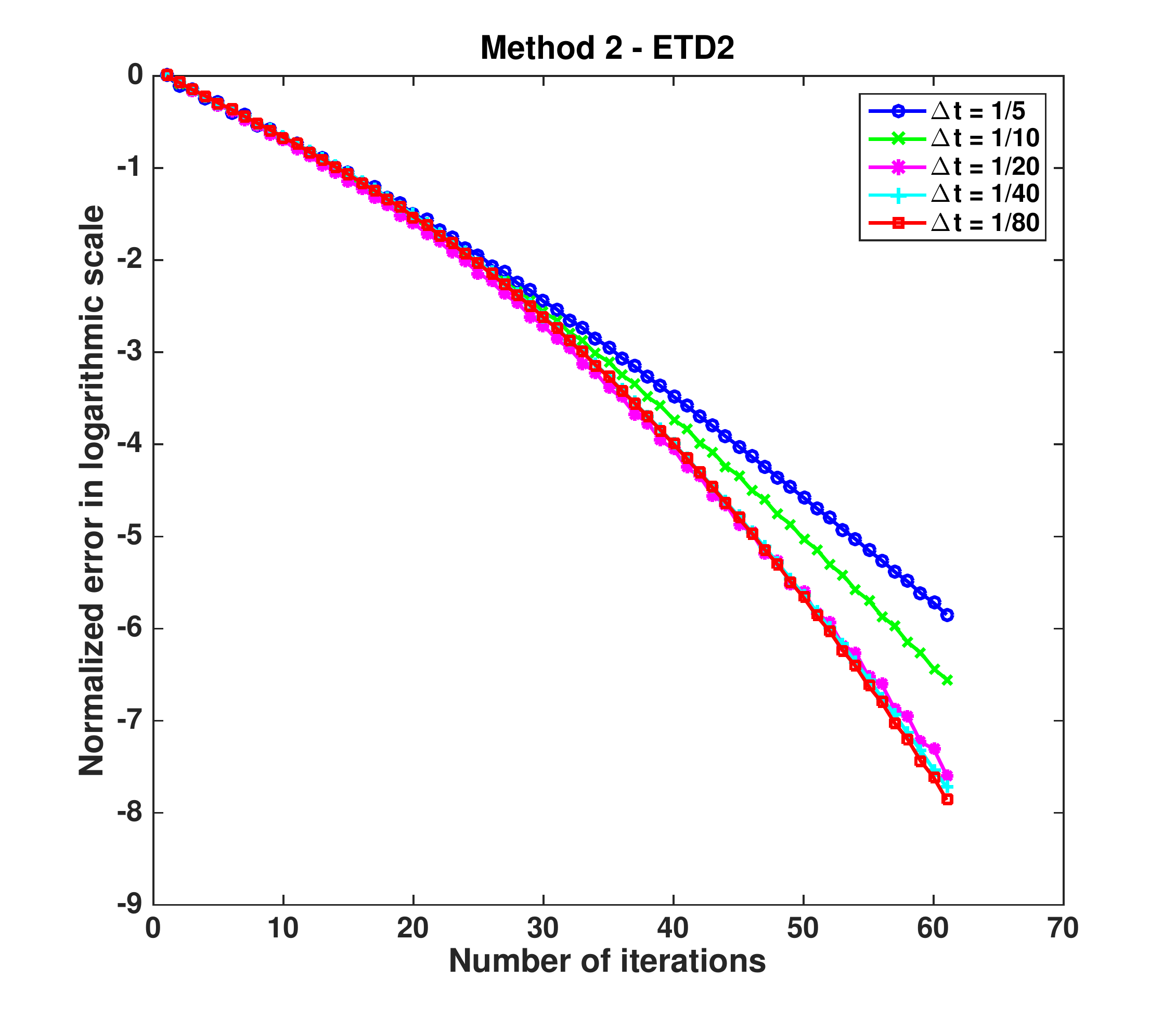}
\end{minipage}
\caption{Decay curves of the normalized $L^{\infty}(0,T, L^{\infty}(\Omega))$ errors of Method 2 over $[0,T]$ for different time step sizes, with the  ETD1 (left) or  the ETD2 (right).} \label{fig:M3.dt} \vspace{-0.6cm}
\end{figure}

\paragraph{\textbf{Convergence vs. different $T$}}

To see the super-linear convergence regime of Schwarz iteration of Method 2, we fix the  overlap size $\delta = 8h$ and the time step size $\Delta t=0.01$ and 
show the error evolution curves for different $T$ $\in$ $\{ 0.25, 0.5, 1, 2, {4}\}$ in Figure \ref{fig:M3.dt2}.  As predicted by the theory, if the time interval becomes larger, the convergence rate becomes linear. To take advantage of the super-linear convergence when a long time interval $[0,T]$ is considered, one should first partition $[0,T]$ into sub-intervals of smaller sizes, called time windows, and then perform Schwarz iteration on each time window (successive time windows do not overlap in time). In addition, for the global-in-time approach, it seems that the ETD2 and ETD1 have quite similar decay rates along the iterations.
\begin{figure}[!htbp]
\centering
\begin{minipage}{0.45 \linewidth}
\centering
 \includegraphics[scale=0.24]{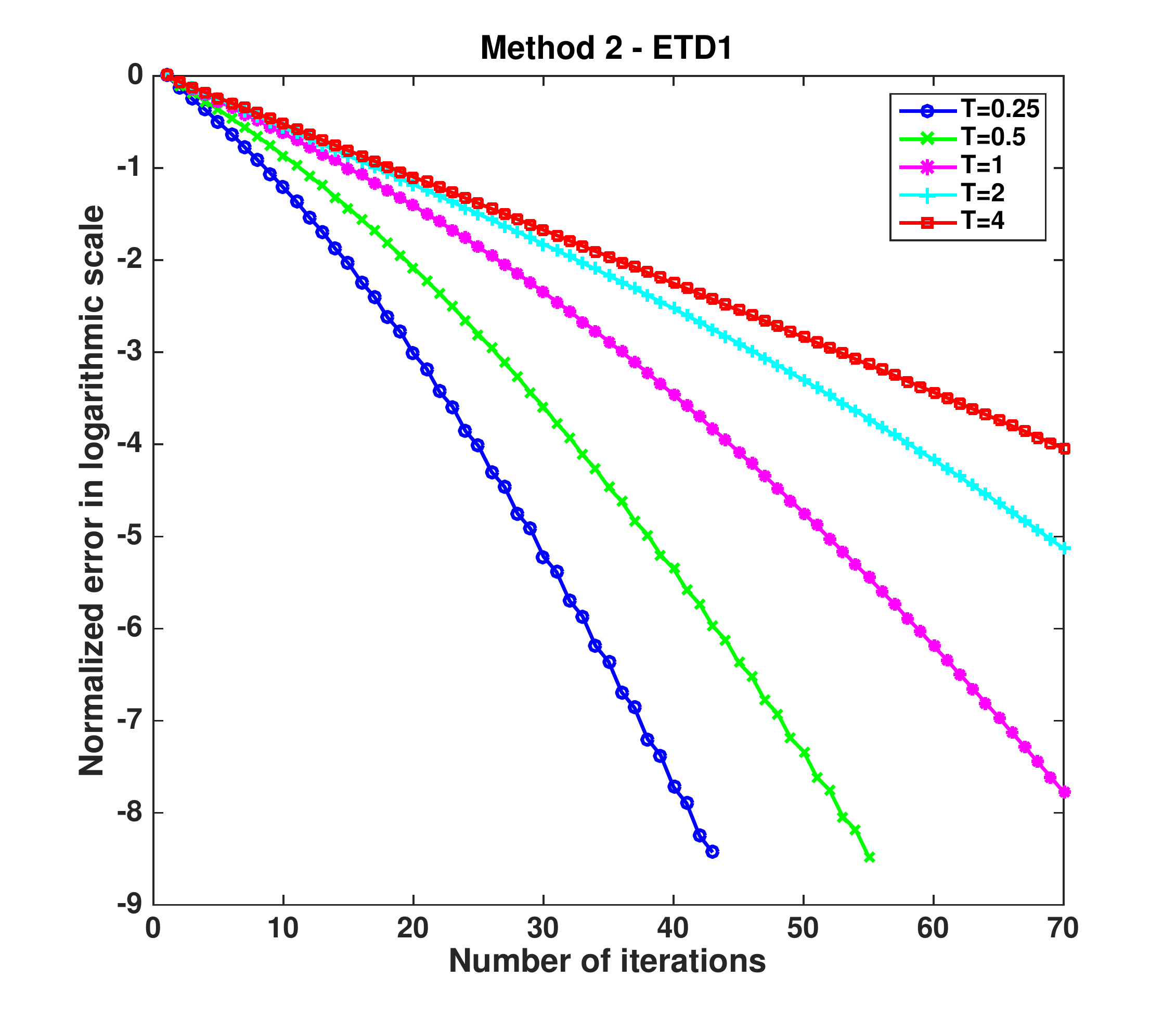}
\end{minipage} \hspace{5pt}
\begin{minipage}{0.45 \linewidth}
\centering
\includegraphics[scale=0.24]{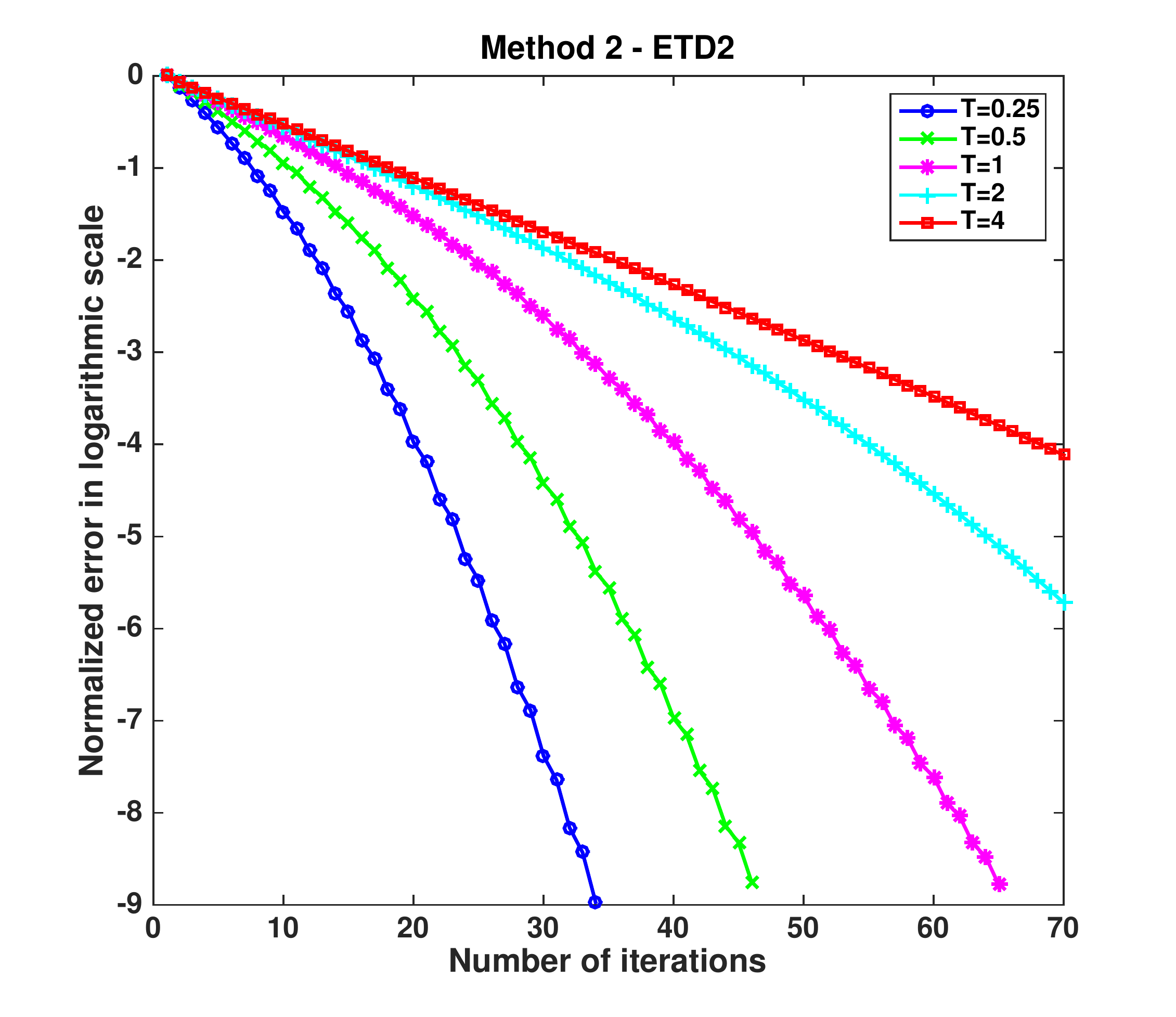}
\end{minipage}
\caption{Decay curves of the normalized $L^{\infty}(0,T, L^{\infty}(\Omega))$ errors of Method 2 over $[0,T]$  for different  $T \in \{ 0.25, 0.5, 1, 2, {4} \}$, with the  ETD1 (left) or  the ETD2 (right). } \label{fig:M3.dt2} \vspace{-0.6cm}
\end{figure}

\paragraph{\textbf{Convergence vs different numbers of subdomains}}

The spatial domain $\Omega=[0,2]$ is split into $P$ non-overlapping uniform subdomains $\widetilde{\Omega}_{i}$. Then each boundary point of $\widetilde{\Omega}_{i}$ interior to the domain $\Omega$ is enlarged by a distance $\delta \in (0,1)$ to form overlapping subdomains $\Omega_{i}$ with a uniform size of overlap equal to $2\delta$. We fix $T = 0.25$, $\Delta t=0.01$, $\delta=4h$ with $h = 2/512 \approx 0.0039$ in this test. We increase the number of subdomains and see {its effects on} the convergence speed of the Schwarz iteration. The results of error decay curves are shown in 
Figures~\ref{fig:numsubM2} and \ref{fig:numsubM3} for $P\in \{2,4,8,16\}$. 
We see that the convergence deteriorates as the number of subdomains increases, and the use of ETD2 helps reduce this deterioration. Note that this is a well-known behavior of domain decomposition methods and a coarse mesh  then often can be additionally used to help obtain convergences independence of the number of subdomains  \cite{ChanMathew94}.

\begin{figure}[!htbp]
\centering
\begin{minipage}{0.45 \linewidth}
\centering
\includegraphics[scale=0.24]{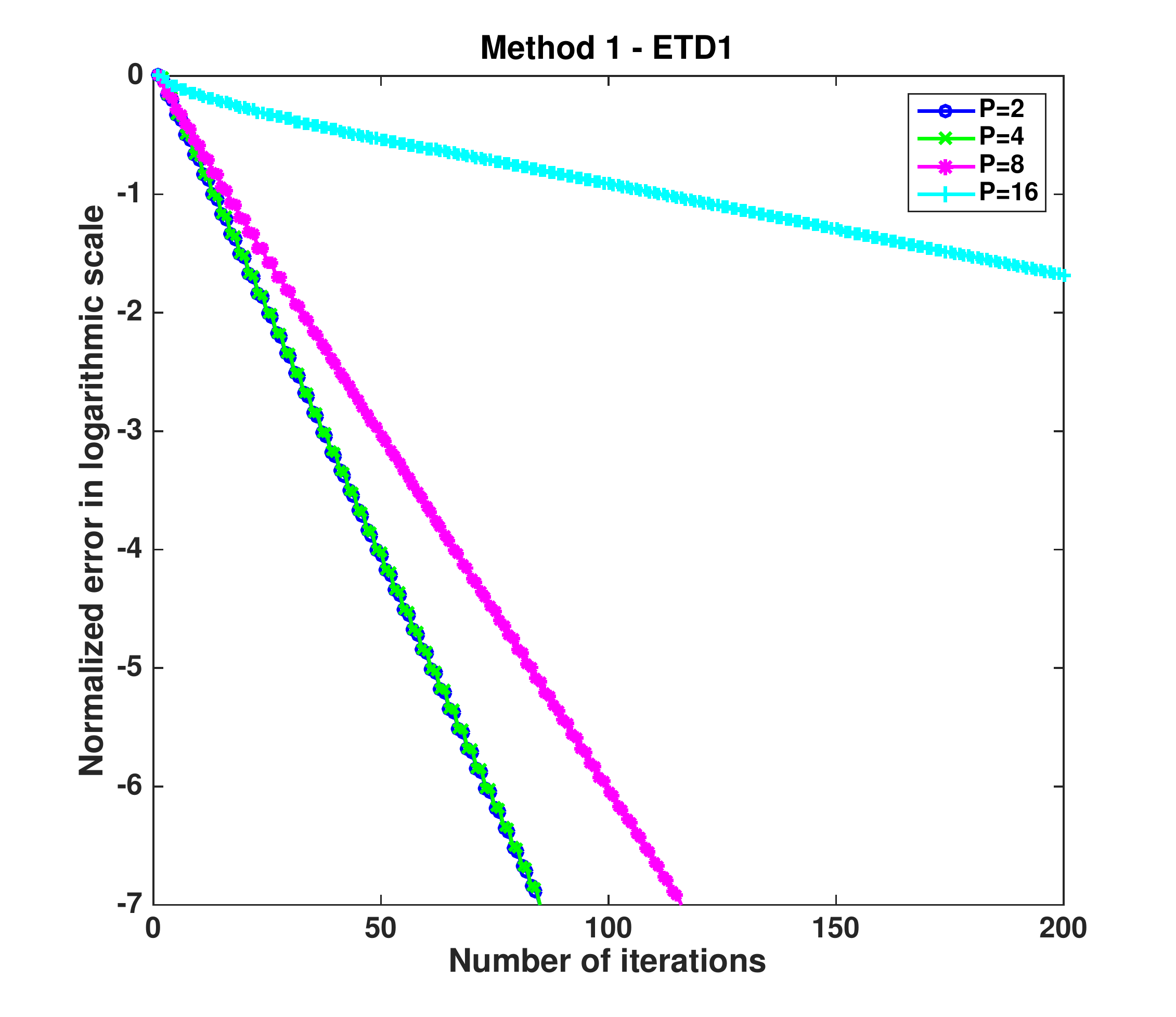}
\end{minipage} \hspace{5pt}
\begin{minipage}{0.45 \linewidth}
\centering
\includegraphics[scale=0.24]{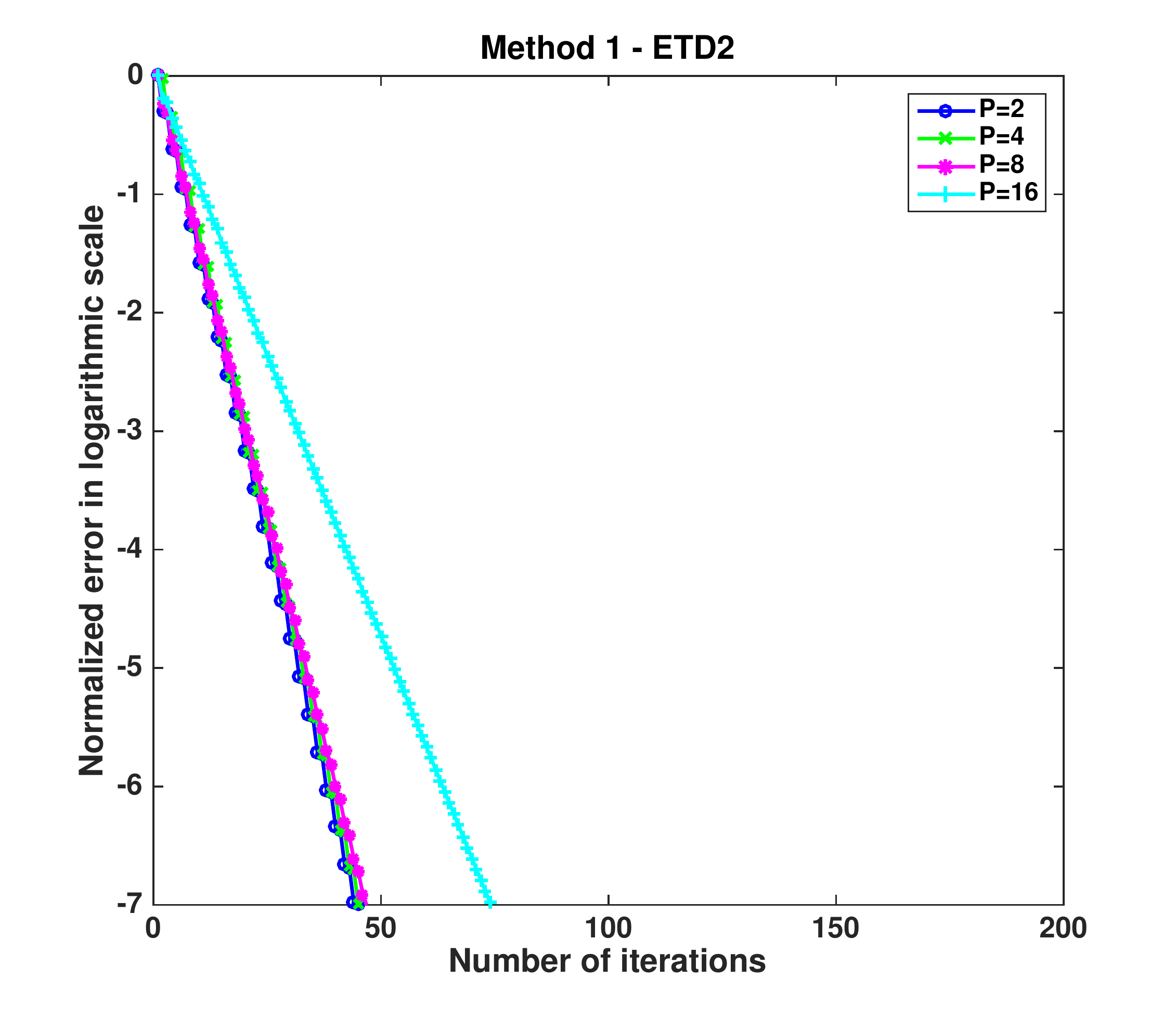}
\end{minipage}
\caption{Decay curves of the normalized  $L^{\infty}(\Omega)$ errors of Method 1 at $t=\Delta t$ for different numbers ($P$) of subdomains, with the  ETD1 (left) or  the ETD2 (right).} \label{fig:numsubM2} \vspace{-0.6cm}
\end{figure}

\begin{figure}[!htbp]
\centering
\begin{minipage}{0.45 \linewidth}
\centering
 \includegraphics[scale=0.24]{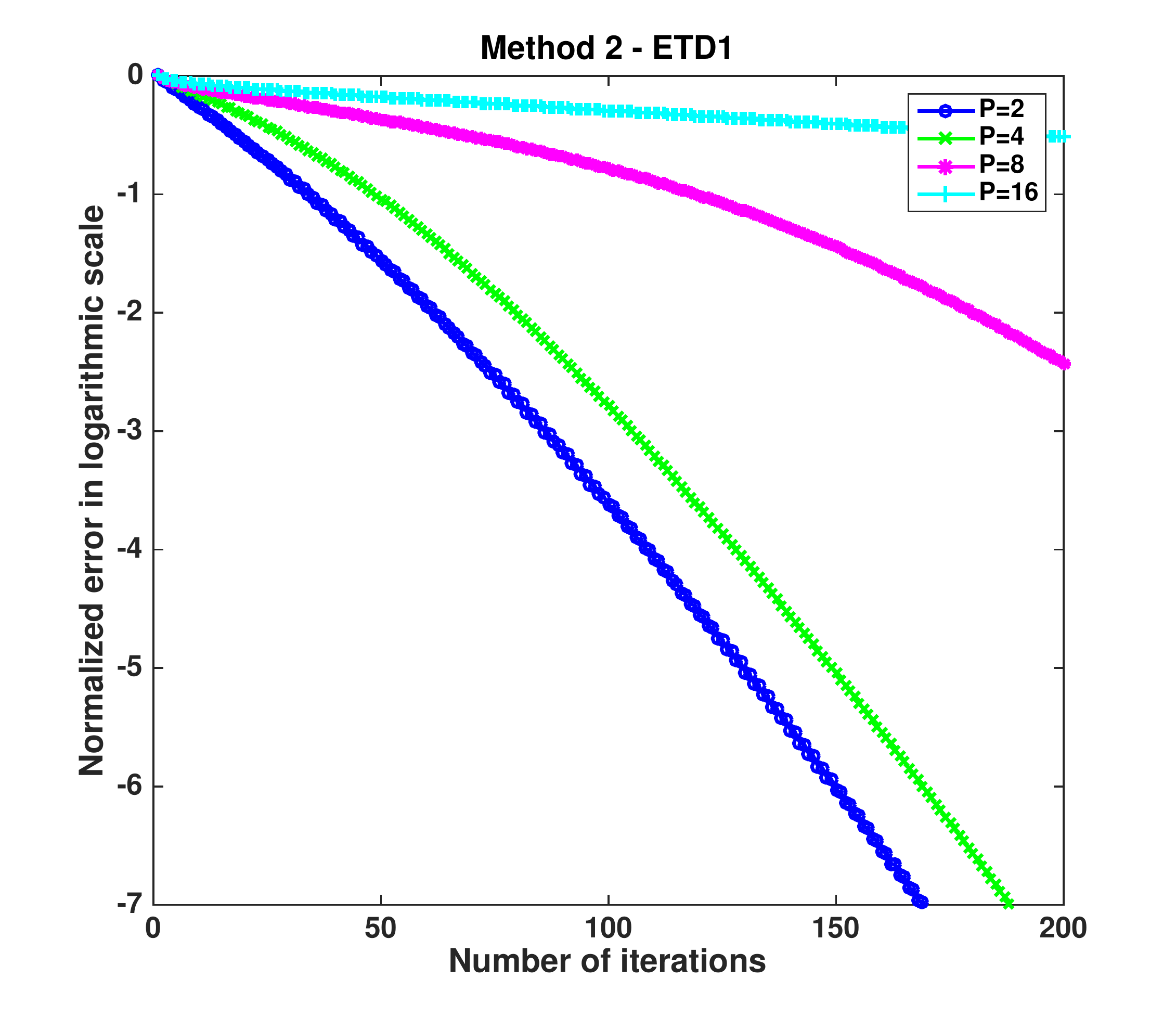}
\end{minipage} \hspace{5pt}
\begin{minipage}{0.45 \linewidth}
\centering
\includegraphics[scale=0.24]{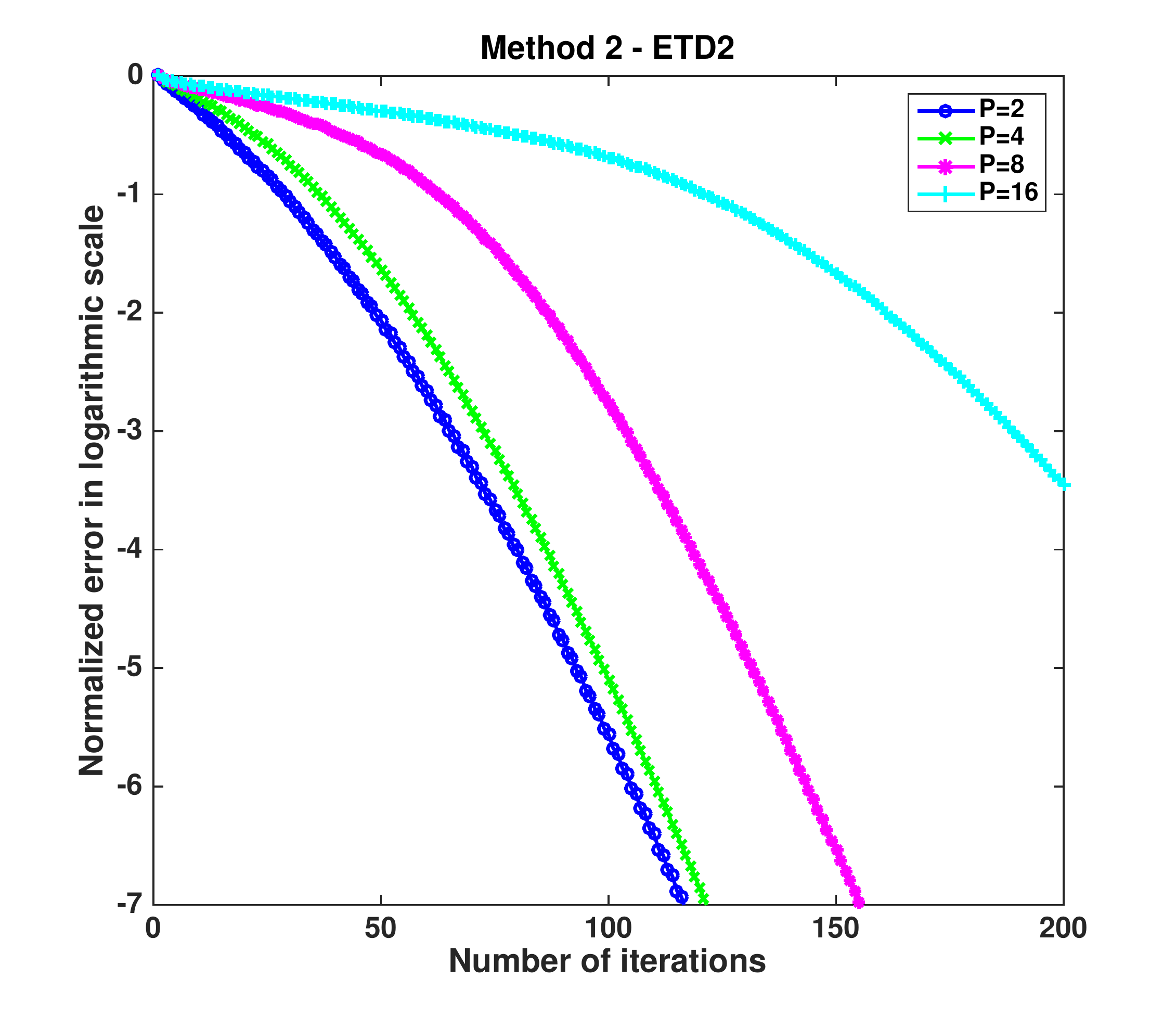}
\end{minipage}
\caption{Decay curves of the normalized $L^{\infty}(0,T, L^{\infty}(\Omega))$ errors of Method 2 over $[0,T]$ for different numbers ($P$) of subdomains, with the  ETD1 (left) or  the ETD2 (right). } \label{fig:numsubM3}
\end{figure}

%
%
\subsection{A 1D example with an analytical solution for testing the accuracy in time of multidomain localized ETD solutions} \label{subsec:anal.test}

Consider the spatial domain $\Omega=[-1,1]$, which is {split}  into two overlapping subdomains $\Omega_{1} = [-1, \delta]$ and $\Omega_{2}=[-\delta, 1]$ for $0 < \delta <1 $. We solve the problem
\begin{equation*}
\begin{array}{rll}\frac{\partial u}{\partial t} &= \frac{\partial^{2} u}{\partial x^{2}} + 2 \pi^{2}  \text{e}^{\pi^{2}t} \sin \left (\pi\big(x-\dfrac{1}{4}\big)\right ), & -1 \leq x \leq 1,\; 0 < t<0.25, 
\end{array} 
\end{equation*}
with the exact solution given by 
$$ u(x,t)= \text{e}^{\pi^{2}t} \sin \left (\pi(x-\frac{1}{4})\right ).
$$
The nonhomogemeous Dirichlet boundary conditions and the initial condition are then determined correspondingly from the exact solution. We fix the mesh size $h = 2/512 \approx 0.0039$, and vary $\Delta t \in \{1/40, 1/80, 1/160, 1/320\}$ and $\delta \in \{ h, 2 h, 4 h, 8 h, 16 h \}$. We would like to verify the temporal accuracy  of the two localized ETD methods. 
For both cases, the converged multidomain solution is defined whenever the relative residual is smaller than a given tolerance $\varepsilon$: $\varepsilon = 10^{-4}$ if the ETD1 is used and $\varepsilon = 10^{-6}$ if the ETD2 is used. 
The relative errors in $L^{\infty}(0,T, L^{\infty}(\Omega))$-norm between the multidomain localized ETD solutions (by \eqref{ETD1Multi} or \eqref{ETD2Multi}) and the exact solution are computed and presented in Tables~\ref{tab:ETD1} and \ref{tab:ETD2} for the localized ETD1  and the localized ETD2  respectively, where the numbers in brackets are the convergence rate of the errors at two successive time step refinement levels. We note that 
once completely converged, the multidomain localized ETD solutions computed by the two iterative domain decomposition algorithms, Method 1 and Method~2, are the same. 

\begin{table}[!htbp]
\centering
\scriptsize
\setlength{\extrarowheight}{3pt}
\begin{tabular}{|c | c | c | c| c| c| }
\hline 
\multicolumn{2}{|c|}{Method}	& \multicolumn{4}{|c|}{Time step size $\Delta t$} \\
\cline{3-6} 
\multicolumn{2}{|c|}{}& $1/40$ & $1/80$ & $1/160$ & $1/320$ \\ 
\hline\hline
\multicolumn{2}{|c|}{Global ETD1} & $1.22E-01$ &	$6.21E-02$ $(0.93)$&	$3.09E-02$ $(0.97)$ &	$1.54E-02$ $(1.00)$ \\
\hline	
 & $\delta =h$ & $3.83E-01$	&$2.46E-01$ $(0.64)$	&$1.60E-01$ $(0.62)$	&$1.04E-01$	$(0.61)$ \\  
\cline{2-6}
\multirow{1}{*}{Localized}& $\delta  =2h$ & $3.73E-01$&	$2.36E-01$ $(0.66)$&	$1.51E-01$ $(0.65)$&	$9.62E-02$ $(0.65)$ \\ 
\cline{2-6}
\multirow{1}{*}{ETD1 }& $\delta  = 4 h$ & $3.53E-01$  &	$2.18E-01$  $(0.70)$ &	$1.34E-01$  $(0.70)$ & $	8.18E-02$ $(0.71)$ \\ 
\cline{2-6}
& $\delta = 8 h$ & $3.17E-01$ &	$1.87E-01$ $(0.77)$&	$1.08E-01$ $(0.79)$&	$6.05E-02$ $(0.84)$ \\ 
\cline{2-6}
& $\delta  = 16 h$ & $2.61E-01$ & 	$1.43E-01$ $(0.86)$ &	$7.62E-02$ $(0.91)$	& $3.93E-02$	$(0.96)$ \\ 
\hline 
\end{tabular}
\caption{Relative errors and convergence rates of the two-subdomain localized ETD1 solutions.} \label{tab:ETD1}  
\end{table}

\begin{table}[!htbp]
\centering
\scriptsize
\setlength{\extrarowheight}{3pt}
\begin{tabular}{|c | c | c | c| c| c| }
\hline 
\multicolumn{2}{|c|}{Method}	& \multicolumn{4}{|c|}{Time step size $\Delta t$} \\
\cline{3-6} 
\multicolumn{2}{|c|}{}& $1/40$ & $1/80$ & $1/160$ & $1/320$ \\ 
\hline\hline
\multicolumn{2}{|c|}{Global ETD2} & $5.17E-03$  & $1.28E-03$ $(2.01)$&	$3.21E-04$ $(2.00)$&	$8.46E-05$ $(1.93)$ 
\\ \hline	
& $\delta =h$ & $1.81E-02$ &	$6.40E-03$ $(1.50)$&	$2.22E-03$ $(1.53)$&	$7.58E-04$ $(1.55)$ 
\\ \cline{2-6} 
\multirow{1}{*}{Localized}& $\delta  =2h$ & $1.74E-02$ &	$6.03E-03$	$(1.53)$ &$2.03E-03$  $(1.57)$&	$6.67E-04$ $(1.61)$ 
\\ \cline{2-6} 
\multirow{1}{*}{ETD2}& $\delta  = 4 h$ & $1.62E-02$ &	$5.37E-03$ $(1.59)$&	$1.71E-03$ $(1.65)$&	$5.21E-04$ $(1.72)$ 
\\ \cline{2-6}
& $\delta = 8 h$ &  $1.41E-02$	& $4.34E-03$ $( 1.70)$&	$1.26E-03$ $(1.79)$&	$3.44E-04$ $(1.97)$ 
\\ \cline{2-6}
& $\delta  = 16 h$ & $1.11E-02$ &	$3.11E-03$ $(1.84)$&	$8.20E-04$ $(1.92)$&	$2.14E-04$ $(1.94)$ 
\\ \hline  
\end{tabular}
\caption{Relative errors and convergence rates of the two-subdomain localized ETD2 solutions.} \label{tab:ETD2} \vspace{-0.7cm}
\end{table}

Note that the errors given by the monodomain (global) ETD method and by the localized ETD  methods are different, which is consistent with the theory as the iterative multidomain solution doesn't converge to the fully discrete monodomain solution, but to the fully discrete multidomain solution (see Section~\ref{sec:Convergence}). Hence, the iterative multidomain solutions corresponding to different  sizes of overlap are not exactly the same. We observe that the orders of the schemes are well preserved if the overlap size is large enough. The errors given by the multidomain solutions are usually larger than those by the monodomain ETD methods, except when sufficiently large overlaps and small time step sizes are used. 
%
%
%
\subsection{A 2D example} \label{subsec:2D}

The spatial domain is $\Omega=[0,\pi]^{2}$, $T=0.5$ and the exact solution is chosen to be  
$$u(x,y,t)=\textstyle\text{e}^{-4t} \sin(x-\frac14) \sin(2(y-\frac18)). $$
The nonhomogemeous Dirichlet boundary conditions and the initial condition are again determined correspondingly from the exact solution.
In space, we use a Cartesian grid with $h = \pi/128$; in time, we use a uniform time step size $\Delta t = T/128$. We consider a decomposition of $\Omega$ into overlapping squares of equal size with a fixed overlap size equal to $9h$. We vary the number of subdomains, and apply Method~1 and Method~2 with the ETD2. 
{The ``converged'' multidomain localized ETD solutions are computed after some fixed number of Schwarz iterations and  compared with the exact solution. Table~\ref{tab:2Dtest} reports the errors between the approximate multidomain solutions and the exact solution in $L^{\infty}(\Omega)$-norm at time $t= T$ under different numbers of 
 subdomains (a total of $P\times P$ subdomains with uniform partition in each direction). The corresponding numbers of iterations are listed in brackets.}
\begin{table}[!htbp]
\centering
\scriptsize
\setlength{\extrarowheight}{3pt}
\begin{tabular}{|c | l | l | l | l | }
\hline 
\# of Subdomains & $1 \times 1$  & $2 \times 2$ & $3 \times 3$ & $4 \times 4$ \\ \hline\hline
Method 1 & \multirow{3}{*}{$2.7910E-03$} & $2.7910E-03$  [2]& $2.7912E-03$ [3]& $2.7906E-03$ [4]\\ \cline{1-1}\cline{1-1}\cline{3-5} \cline{3-5} 
\multirow{2}{*}{Method 2} & & $2.4073E-01$ [2] & $3.4382E-01$ [3]& $3.2665E-01$ [4]\\ \cline{3-5} 
 & &$2.7913E-03$ [14]& $2.7931E-03 $ [19] & $2.7911E-03$ [23] \\ \hline
\end{tabular}
\caption{$L^{\infty}(\Omega)$ errors at time $t= T$ between the approximate multidomain localized ETD solutions (using the ETD2) and the exact solution; the numbers of Schwarz iterations  used are shown in brackets. } \label{tab:2Dtest} \vspace{-0.4cm}
\end{table}

It can be seen that for a sufficiently large size of overlap, Method~1 converges after a few iterations (just $P$, {the numbers of subdomains in one direction}) despite the number of subdomains and reaches the accuracy of  the monodomain ETD solution. However, for Method~2, the convergence is slower. {It takes more iterations to achieve the desired accuracy.} In particular, if the number of iterations is fixed to be $P$, the numerical errors are much larger than the error given by the monodomain ETD solution.  At least for this example with conforming time step sizes, Method~1 seems more efficient than Method~2. 

\section{Conclusions}\label{cons}
In this paper, we have introduced two iterative, localized exponential time differencing methods based on overlapping domain decomposition for the time-dependent diffusion equation: Method 1 with iterations at each time step and Method 2 in which time dependent problems are solved at each iteration.  Convergence analysis is rigorously studied for the one-dimensional (in space) case with discussions of extensions to higher-dimensional problems. Numerical experiments in 1D and 2D spaces confirm that both iterative domain decomposition algorithms  converge linearly  (at each time step or the whole time window) and the convergence rate depends on the size of overlap. For Method 1, the convergence rate is dependent on the time step size as well. For Method 2 with short time windows, it could converge super-linearly. Since Method 2 is global in time, it makes possible the use of different time steps in the subdomains according to their physical properties. To accelerate the convergence, one should use short time intervals (called time windows) and use the solution in the previous time window to calculate a ``good initial'' guess on the space-time interface. 

\bibliographystyle{siam}

\end{document}